\documentclass[11pt]{amsart}
\usepackage{amsmath}
\usepackage{amssymb}
\usepackage{amscd}
\def\NZQ{\mathbb}               % the font for N,Z,Q,R,C
\def\NN{{\NZQ N}}
\def\QQ{{\NZQ Q}}
\def\ZZ{{\NZQ Z}}
\def\RR{{\NZQ R}}

\newtheorem{Theorem}{Theorem}[section]
\newtheorem{Lemma}[Theorem]{Lemma}
\newtheorem{Corollary}[Theorem]{Corollary}
\newtheorem{Proposition}[Theorem]{Proposition}
\newtheorem{Remark}[Theorem]{Remark}

\let\epsilon\varepsilon
\let\phi=\varphi
\let\kappa=\varkappa

\begin{document}

\title{Rees algebras and the reduced fiber cone
of divisorial filtrations  on two dimensional normal local rings}
\author{Steven Dale Cutkosky}
\thanks{Partially supported by NSF grant DMS-2348849}

\address{Steven Dale Cutkosky, Department of Mathematics,
University of Missouri, Columbia, MO 65211, USA}
\email{cutkoskys@missouri.edu}


\begin{abstract} 
Let $\mathcal I=\{I_n\}$ be a $\QQ$-divisorial filtration on a two dimensional normal excellent local ring $(R,m_R)$. Let $R[\mathcal I]=\oplus_{n\ge 0}I_n$ be the Rees algebra of $\mathcal I$ and $\tau:\mbox{Proj}R[\mathcal I])\rightarrow \mbox{Spec}(R)$ be the natural morphism. The reduced fiber cone of $\mathcal I$ is the $R$-algebra $R[\mathcal I]/\sqrt{m_RR[\mathcal I]}$, and the reduced exceptional fiber of $\tau$ is $\mbox{Proj}(R[\mathcal I]/\sqrt{m_RR[\mathcal I]})$.
In \cite{CNg}, we showed that in spite of the fact that $R[\mathcal I]$ is often not Noetherian,  $m_RR[\mathcal I]$ always has only finitely many minimal primes, so  $\tau^{-1}(m_R)$ has only finitely many irreducible components. 
%In contrast, we give an example in Section \ref{SecExample} of this paper, showing that for general (non divisorial) filtrations $\mathcal J$ of $R$, it is possible for $m_RR[\mathcal J]$ to have infinitely many minimal primes. 
In Theorem \ref{Theorem6}, we give an explicit description of the scheme structure of $\mbox{Proj}(R[\mathcal I])$. As a corollary, we obtain in Theorem \ref{Theorem11} a new proof of a theorem of F. Russo, showing that $\mbox{Proj}(R[\mathcal I])$ is always Noetherian and that $R[\mathcal I]$ is Noetherian  if and only if $\mbox{Proj}(R[\mathcal I])$ is a proper $R$-scheme.
 In Corollary \ref{Cor7} to Theorem \ref{Theorem6}, we give an explicit description of the scheme structure of the reduced exceptional fiber $\mbox{Proj}(R[\mathcal I]/\sqrt{m_RR[\mathcal I]})$ of $\tau$, in terms of the possible values 0, 1 or 2 of the analytic spread 
$\ell(\mathcal I)=\dim R[\mathcal I]/m_RR[\mathcal I]$. 
In the case that $\ell(\mathcal I)=0$, $\tau^{-1}(m_R)$ is the emptyset; this case can only occur if $R[\mathcal I]$ is not Noetherian. 

At the end of the introduction, we give a simple example of a graded filtration $\mathcal J$ of a two dimensional regular local ring $R$  such that $\mbox{Proj}(R[\mathcal J])$ is not Noetherian. This filtration is necessarily not divisorial. 
\end{abstract}

\maketitle 

\section{Introduction}\label{SecInt}
Let $(R,m_R)$ be a local ring, which we assume to be Noetherian. A graded filtration of $R$ is a filtration $\mathcal I=\{I_n\}$  of $R$ which satisfies $I_mI_n\subset I_{m+n}$ for all $m,n$.

The Rees algebra of a graded filtration $\mathcal I$ of $R$ is
$$
R[\mathcal I]=\oplus_{n\ge 0}I_n.
$$
We will say that $\mathcal I$ is Noetherian if $R[\mathcal I]$ is Noetherian, which holds if and only if $R[\mathcal I]$ is a finitely generated $R$-algebra.

Expositions of the theory of complete ideals, integral closure of ideals and their relation to valuation ideals, Rees valuations, analytic spread and birational morphisms can be found, from different perspectives,  in  \cite{ZS2}, \cite{HS}, \cite{Li2} and \cite{Li3}. The book \cite{HS} and the article \cite{Li3} contain references to  original work in this subject. Concepts in this introduction which are not defined in this section or in these references can be found later in this paper. A survey of recent work on symbolic algebras is given in \cite{DDGHN}. A different notion of analytic spread for families of ideals is given in \cite{DM}.
The recent papers \cite{ORWY}and \cite{OWY} use geometric methods to  undertand ideal theory in two dimensional normal local domains.

In this paper, we study properties of Rees algebras of graded filtrations and properties of the $R$-scheme $\mbox{Proj}(R[\mathcal I])$. In particular, we are interested in when good  properties of Noetherian filtrations hold for NonNotherian filtrations.  
%In Section \ref{SecExample}, we give an example of a graded filtration $\mathcal J=\{J_n\}$ of $m_R$-primary ideals in a two dimensional regular local ring $R$ such that $m_RR[\mathcal J]$ has infinitely many minimal primes $P_i$, all of which satisfy $\dim R[\mathcal I]/P_i=2$. Thus the fiber cone
%$R[\mathcal I]/m_RR[\mathcal I]$ of $\mathcal I$ has infinitely many minimal primes. 

%The situation for divisorial filtrations is remarkably  better, at least on a normal excellent local domain, as we will show in this paper.  

$\QQ$-divisorial filtrations of a domain (which we will call divisorial filtrations in this introduction) are defined in  Section \ref{SecDiv}.  Divisorial filtrations appear naturally in many contexts.
Examples include the filtration of integral closures of powers of an ideal $I$, $\{\overline{I^n}\}$, and the filtration of symbolic powers $\{P^{(n)}\}$ of a prime ideal $P$ in a domain $R$ such that $R_p$ is a regular local ring. The Rees algebra $R[\mathcal I]$ of a divisorial filtration is often not Noetherian. For instance, there are examples of filtrations of symbolic powers of height two prime ideals in a three dimensional regular local ring which are not Noetherian (\cite{Na}, \cite{Rob}, \cite{GNW}). Further, the symbolic filtration of a prime ideal $P$ in a normal two dimensional local ring is Noetherian  if and only if  $P$ is torsion in the class group of $R$. 

We focus on the situation of divisorial filtrations in a two dimensional normal excellent local ring $R$. In general, such $R$ will have lots of divisorial filtrations which  are not Noetherian.  In fact, if $R$ is a complete normal local ring of dimension two whose residue field is algebraically closed of characteristic zero, then all divisorial filtrations of $R$ are Noetherian if and only if  $R$ has a rational singularity.  
By Theorem 12.1 of \cite{Li2}, if $R$ has a rational singularity, then $\mathcal I$ is Noetherian for all divisorial filtrations $\mathcal I$ of $R$, and by the Corollary to Theorem 4 \cite{C2}, if $\mathcal I$ is Noetherian for all divisorial filtrations $\mathcal I$, then $R$ has a rational singularity.

The starting point of this paper is a result in 
Proposition 6.7 \cite{CNg} (recalled in Proposition \ref{Prop1} later in this paper) showing that    
if $\mathcal I$ is a divisorial filtration on a two dimensional normal excellent local ring, then $m_RR[\mathcal I]$ has only finitely many minimal primes. 
This leads to the question of  which other Noetherian like properties $\mathcal I$ will have for an arbitrary divisorial filtration $\mathcal I$? We also ask the related question of what  we can say about $\mbox{Proj}(R[\mathcal I])$, and how it is like or different from the case when $R[\mathcal I]$ is Noetherian?  %The example in Section \ref{SecExample} mentioned above has infinitely many minimal primes, so to have good general properties, we must restrict to divisorial filtrations. 

The analytic spread $\ell(I)$ of an ideal $I$ in a (Noetherian) local ring $(R,m_R)$ is defined to be 
$\ell(I)=\dim R[It]/m_RR[It]$. We have that
\begin{equation}\label{eqI1}
\ell(I)\le\dim R,
\end{equation}
 by Proposition 5.1.6 \cite{HS} and 
 \begin{equation}\label{eqI2}
 \mbox{ht}(I)\le \ell(I)
 \end{equation}
  by Corollary 8.3.9 \cite{HS}.

 The analytic spread is extended to graded filtrations in \cite{CPS}. For a graded filtration $\mathcal I$ on a local ring $R$, we define
 $\ell(\mathcal I)=\dim R[\mathcal I]/m_RR[\mathcal I]$. We have that  
\begin{equation}\label{eqASUB}
\ell(\mathcal I)\le \dim R
\end{equation}
by Lemma 3.6  \cite{CPS}, extending the inequality of (\ref{eqI1}) from ideals to graded filtrations. However, the inequality (\ref{eqI2}) fails completely for filtrations, even for  divisorial filtrations. Theorem 1.11 \cite{CPS} gives an example of a height two prime  ideal $P$ in a three dimensional regular local ring such that the symbolic filtration $\{P^{(n)}\}$ satisfies $\ell(\{P^{(n)}\})=0$, and Example 7.3 \cite{CNg} gives an example of a height one  ideal $I$ in a two dimensional normal excellent local ring $R$ such that $\ell(\{I^{(n)}\})=0$.

\subsection{$\QQ$-divisorial filtrations on two dimensional excellent normal local domains}\label{SubIn}
From now on, we will assume that $(R,m_R)$ is a two dimensional, excellent normal local ring with quotient field $K$, and $\mathcal I=\{I_n\}$ is a $\QQ$-divisorial filtration on $R$. In this situation, we are able to give a complete description of $\mbox{Proj}(R[\mathcal I])$. We first give a quick overview of a construction from \cite{CNg}. More details are given in Section \ref{SecRes}.

There exists  a birational projective morphism 
\begin{equation}\label{eqI5}
\pi:X\rightarrow\mbox{Spec}(R)
\end{equation}
 such that $X$ is nonsingular, and an effective  $\QQ$-divisor $\Delta$ on $X$ such that  $-\Delta$ is nef and 
 \begin{equation}\label{eqI6}
 I_n=\Gamma(X,\mathcal O_X(-\lceil n\Delta\rceil ))
 \end{equation}
  for all $n\ge 0$. We have that 
$\pi:X\setminus \pi^{-1}(m_R)\rightarrow \mbox{Spec}(R)\setminus\{m_R\}$ is an isomorphism.

We have that  $E_1+\cdots+E_r$ is a simple normal crossings divisor on $X$, where $E_1,\ldots,E_r$ are the prime exceptional divisors of $ \pi$.

 The Rees algebra associated to $\mathcal I$, as defined earlier, is   
 $$
 R[\mathcal I]=\oplus_{n\ge 0}I_n=\oplus_{n\ge 0}\Gamma(X,\mathcal O_X(-\lceil n\Delta\rceil)).
  $$
   Define $Z(\mathcal I)=\mbox{Proj}(R[\mathcal I])$, with natural morphism $\tau:Z(\mathcal I) \rightarrow \mbox{Spec}(R)$. We have that 
  $Z(\mathcal I)\setminus \tau^{-1}(m_R) \cong \mbox{Spec}(R)\setminus \{m_R\}$
   by Lemma \ref{LemmaTau}. 
   
   In Proposition 6.7 \cite{CNg} (restated as Proposition \ref{Prop1} later in this paper) it is shown that 
   $P_j=P_{E_j}:=\oplus_{n\ge 0}\Gamma(X,\mathcal O_X(-\lceil n\Delta\rceil-E_j))$, where $E_j$ for $1\le j\le r$ are the exceptional divisors of $\pi$, are prime ideals in $R[\mathcal I]$ and $\sqrt{m_RR[\mathcal I]}=\cap_{i=1}^rP_i$.
  It thus follows that $m_RR[\mathcal I]$ has only finitely many minimal primes, as was stated earlier in this introduction.

In the following theorem, which will be proven in Section \ref{SecAS},
we obtain finer information about the  structure of  the fiber cone $R[\mathcal I]/m_RR[\mathcal I]$ of $\mathcal I$.

\begin{Theorem}\label{Theorem2} Let $R$ be a two dimensional excellent local ring and $\mathcal I$ be a $\QQ$-divisorial filtration on $R$, so that $\ell(\mathcal I)\le \dim R=2$.
 The following occurs in the three possible cases $0,1,2$ of analytic spread $\ell(\mathcal I)$ of $\mathcal I$.
\begin{enumerate} 
\item[1)] If $\ell(\mathcal I)=0$, then $\sqrt{m_RR[\mathcal I]}=m_R\oplus R[\mathcal I]_+$ and $R[\mathcal I]$ is not a finitely generated $R[\mathcal I]$-algebra.
\item[2)] If $\ell(\mathcal I)=1$, then $R[\mathcal I]$ is a finitely generated $R$-algebra and there is a unique exceptional curve $E$ such that $\sqrt{m_RR[\mathcal I]}=P_E$  with $\dim R[\mathcal I]/P_E=1$. 
\item[3)] If $\ell(\mathcal I)=2$, then $\sqrt{m_RR[\mathcal I]}=\cap P_{E_i}$ where the intersection is over all the prime ideals $P_{E_i}$ such that 
$\dim R[\mathcal I]/P_{E_i}=2$. 
\end{enumerate}
\end{Theorem}

%By Remark \ref{Thm2RM}, the prime ideals $P_{E_i}$ of the intersection $\sqrt{m_RR[\mathcal I]}=\cap P_{E_i}$
%in the case $\ell(\mathcal I)=2$ of Theorem \ref{Theorem2} are all distinct.  

All three cases of analytic spread do occur. 
As mentioned earlier in the introduction, Example 7.3 \cite{CNg} gives an example where $\ell(\mathcal I)=0$. $\ell(\mathcal I)=1$ if $\mathcal I$ is the symbolic filtration of a height one prime ideal $P$ which is torsion in the class group of $R$ (for example if $P$ is a principal ideal). Every divisorial filtration of $m_R$-primary ideals $\mathcal I$ in $R$ satisfies $\ell(\mathcal I)=2$ by Theorem 1.3 \cite{CPS}.

The concept of the stable base locus of a line bundle on a projective variety is discussed in \cite{LA}. In Section \ref{SecsBL}, we develop it in our more local context. 

Let $\mathcal L$ be an invertible sheaf on $X$. Define the base locus of $\Gamma(X,\mathcal L)$ to be 
$$
\mbox{BL}(\Gamma(X,\mathcal L))=\{p\in X\mid \Gamma(X,\mathcal L)\mathcal O_{X,p}\ne \mathcal L_p\}.
$$
This is a closed subset of $X$ (regarded as an algebraic set and not as a scheme). The stable base locus of $\mathcal L$ is
$\mbox{sBL}(\mathcal L)=\cap_{n\ge 1}\mbox{BL}(\Gamma(X,\mathcal L^n))$,
which is also a closed subset of $X$. By Lemma \ref{Lemma4} (whose proof is essentially the same as the proof of Proposition 2.1.21 \cite{LA})
there exists $n_1$ such that 
$$
\mbox{BL}(\Gamma(X,\mathcal L^{nn_1}))=\mbox{sBL}(\mathcal L)
$$ for all $n\ge n_1$.

We will denote $\mbox{sBL}(\mathcal O_X(-n\Delta))$ by $\mbox{sBL}(-n\Delta)$ when $n\Delta$ is an integral divisor.

In the final section, Section \ref{SecNoeth}, we determine the scheme structure of $Z(\mathcal I)=\mbox{Proj}(R[\mathcal I])$. In Subsection \ref{SubProof}, we establish the following theorem, which shows that $Z(\mathcal I)$ has a remarkably simple scheme structure.

\begin{Theorem}\label{Theorem6} Let $R$ be a two dimensional normal excellent local ring and $\mathcal I$ be a $\QQ$-divisorial filtration in $R$.
Let $\pi:X\rightarrow \mbox{Spec}(R)$ and $\Delta$ be defined as in (\ref{eqI5}) and (\ref{eqI6}), so that $I_n=\Gamma(X,\mathcal O_X(-\lceil n\Delta \rceil))$ for $n\ge 0$.
Let $n_0$ be the smallest positive integer $n$ such that $n\Delta$ is a $\ZZ$-divisor and 
$$
\mbox{BL}(\Gamma(X,\mathcal O_X(-mn\Delta)))=\mbox{sBL}(-n\Delta)\mbox{ for all $m\ge 1$}. 
$$
Then the  exceptional curves $E$ of $\pi:X\rightarrow \mbox{Spec}(R)$ which satisfy the following condition:
\begin{equation}\label{eqTh6}
\mbox{$E$ is not contained in $\mbox{sBL}(-n_0\Delta)$ and $(E\cdot \Delta)=0$}
\end{equation}
are disjoint from $\mbox{sBL}(-n_0\Delta)$. There is a natural birational proper morphism 
$$
X\setminus\mbox{sBL}(-n_0\Delta)\rightarrow \mbox{Proj}(R[\mathcal I])
$$
 which is obtained by contracting  all exceptional curves $E$ satisfying (\ref{eqTh6}). In particular, $\mbox{Proj}(R[\mathcal I])$ is a finite type, normal, integral $R$-scheme which is birational to $\mbox{Spec}(R)$.
\end{Theorem}

As a corollary of our explicit description of $\mbox{Proj}(R[\mathcal I])$, we obtain a new proof of the following  theorem of Russo.

\begin{Theorem}\label{Theorem11}(F. Russo, Theorem 4 \cite{Ru})  Let $R$ be a two dimensional normal excellent local ring and $\mathcal I$ be a $\QQ$-divisorial filtration in $R$. Then $\mbox{Proj}(R[\mathcal I])$ is a normal, finite type $R$-scheme, and thus is  Noetherian. $R[\mathcal I]$ is a finitely generated $R$-algebra if and only if $\mbox{Proj}(R[\mathcal I])$ is a proper $R$-scheme. 
\end{Theorem}

\begin{proof} The first statement,  on the Noetherianity of $\mbox{Proj}(R[\mathcal I])$, follows immediately from Theorem \ref{Theorem6}. The second statement, giving a condition for  finite generation of $R[\mathcal I]$, follows from Theorem \ref{Theorem6} and Lemma \ref{Lemma11}.
\end{proof}

The proof in \cite{Ru} makes ingenious use of a relative form of a Theorem of Zariski in \cite{Z3} about divisors on a projective surface. In \cite{Ru}, it is proven that if $D$ is an effective divisor on $X$, then $\Gamma(X\setminus \mbox{Supp}(D),\mathcal O_X)$ is a finitely generated $R$-algebra. 
We give a brief sketch how this is used in \cite{Ru} to prove Theorem \ref{Theorem11}. After replacing $\Delta$ with a positive multiple of $\Delta$, we may assume that $\Delta$ is an integral divisor. 
For $H\in R[\mathcal I]$ homogeneous of positive degree, there is a natural morphism $X_H\rightarrow \mbox{Spec}(R[\mathcal I]_{(H)})$, where $X_H=\{q\in X\mid H(q)\ne 0\}$ (regarding $H$ as an element of $\Gamma(X,\mathcal O_X(-n\Delta))$). 
Now the union $\cup X_H$ over such $H$ is $X\setminus \mbox{sBL}(-\Delta)$, and  we obtain the  ``Grothendieck map'' $r:X\setminus \mbox{sBL}(-\Delta)\rightarrow \mbox{Proj}(R[\mathcal I])$. Since $X\setminus \mbox{Supp}(D)$ is Noetherian and quasi compact and since by the local Zariski theorem each $\Gamma(X_H,\mathcal O_X)$ is a finite type $R$-algebra we deduce that $r$ is a proper birational morphism, so it is in particular surjective, and that $\mbox{Proj}(R[\mathcal I])$ is a finite type $R$-scheme. 

Theorem \ref{Theorem4}, which is an essential part of the proof of  Theorem \ref{Theorem6}, could be deduced starting from assuming the conclusions of Theorem \ref{Theorem11}, instead of giving an independent proof of Theorem \ref{Theorem11} as we do in this paper. This would lead to a slight  simplification of the proof of Theorem \ref{Theorem4}.

Combining Theorem \ref{Theorem2}  and Theorem \ref{Theorem6}, we obtain the following corollary, giving a much stronger formulation of Theorem \ref{Theorem2}.

\begin{Corollary}\label{Cor7} The reduced exceptional fiber 
$$
\tau^{-1}(m_R)=\mbox{Proj}(R[\mathcal I]/\sqrt{m_RR[\mathcal I]})\subset Z(\mathcal I)
$$
is a quasi-projective $R/m_R$-scheme, which 
has the following forms, in the three cases of analytic spread $\ell(\mathcal I)$.
\begin{enumerate}
\item[1)] If $\ell(\mathcal I)=0$, then $\tau^{-1}(m_R)=\emptyset$.
\item[2)] If $\ell(\mathcal I)=1$, then $\tau^{1}(m_R)$ is a point which is finite over $R/m_R$.
\item[3)] If $\ell(\mathcal I)=2$, then $\tau^{-1}(m_R)$ is an open subset of a reduced projective $R/m_R$-scheme of pure dimension one (all  irreducible components have dimension one).
\end{enumerate}

\end{Corollary}

We will say that a valuation $v$ of $K$ dominates a local ring $(A,m_A)$ contained in $K$ if the valuation ring $\mathcal O_v$ of $v$ dominates $A$; that is, 
 $A\subset \mathcal O_v$ and $m_v\cap \mathcal A=m_A$. If $X$ is an  integral scheme such that $K=\mathcal O_{X,\eta}$ where $\eta$ is the generic point of $X$, and $v$ is a valuation of $K$, we will say that a point $\alpha\in X$ is a center of $v$ on $X$ if $v$ dominates $\mathcal O_{X,\alpha}$.

The following invariant, which is defined and discussed in Section \ref{SecInv}, plays an essential role in the proof of Theorem \ref{Theorem6}.
Let $v$ a valuation of $K$ which is nonnegative on $R$. 
%The facts that
%$I_m^n\subset I_{mn}$ and $v(I_m^n)=nv(I_m)$ implies
%$$
%\frac{v(I_{mn})}{mn}\le \frac{nv(I_m)}{mn}=\frac{v(I_m)}{m},
%$$
%and so 
%\begin{equation}\label{eq2}
%\frac{v(I_{mn})}{mn}\le
%\min\{\frac{v(I_m)}{m},\frac{v(I_n)}{n}\}.
%\end{equation}
We define
$$
\gamma_v(\mathcal I)=\inf\{\frac{v(I_n)}{n}\}
$$
where the infimum is over all positive $n$.

We have that  if $\frac{v(I_m)}{m}=\gamma_v(\mathcal I)$ for some $m$, then
$$
\frac{v(I_{mn})}{mn}=\gamma_v(\mathcal I)\mbox{ for all $n>0$.}
$$

We obtain the following theorem, which is proven at the end of Subsection \ref{SubProof}.
\begin{Theorem}\label{Theorem7} Let $n_0$ be the smallest positive integer such that $-n_0\Delta$ is a $\ZZ$-divisor and $\mbox{BL}(\Gamma(X,\mathcal O_X(-nn_0\Delta)))=\mbox{sBL}(-n_0\Delta)$ for all $n\ge 1$.

Suppose that $v$ is a valuation of $K$ which is nonnegative on $R$. Then the following  are equivalent:
\begin{enumerate}
\item[1)] $v$ has a center on $Z(\mathcal I)=\mbox{Proj}(R[\mathcal I])$.
\item[2)] The center of $v$ on $X$ is in $X\setminus\mbox{sBL}(-n_0\Delta)$.
\item[3)] There exists a positive integer $m$ such that $\frac{v(I_m)}{m}=\gamma_v(\mathcal I)$.
\end{enumerate}
\end{Theorem}

If $\mathcal J$ is a graded filtration of a two dimensional normal excellent ring $R$, it may be that $\mbox{Proj}(R[\mathcal J])$ is not Noetherian. Here is a simple example. Let $R=k[[x,y]]$ be a power series ring in two variables over a field $k$, and let $\mathcal J=\{J_n\}$ where $J_n=(x,y^n)$. Writing $R[\mathcal J]=\sum_{n\ge 0}J_nt^n$, where $t$ is an indeterminate, we see that 
$$
R[\mathcal J]_{(ty)}=R[\frac{x}{y^n}\mid n\in \ZZ_{>0}], 
$$
which is not Noetherian. Since $\mbox{Spec}(R[\mathcal J]_{(ty)})$ is an affine open subset of $\mbox{Proj}(R[\mathcal J])$, $\mbox{Proj}(R[\mathcal J])$ is not Noetherian. 

\section{Integral closure}
Integrally closed ideals and 
the integral closure $\overline I$ of an ideal $I$ are defined in Chapter 1 of \cite{HS}. We will need the following well known lemma. 

\begin{Lemma}\label{LemmaC} Let $R$ be a normal excellent local ring with quotient field $K$. Then
\begin{enumerate}
\item[1)] Let $I\subset R$ be an ideal and $X$ be the normalization of the blowup of $I$. Then the natural morphism $\beta:X\rightarrow \mbox{Spec}(R)$ is a birational projective morphism since $R$ is universally Nagata. Let $\pi:Y\rightarrow \mbox{Spec}(R)$ be a birational projective morphism which factors through $X$. Then $\Gamma(Y,I\mathcal O_Y)=\overline I$ is the integral closure of $I$.
\item[2)] Let $W\rightarrow \mbox{Spec}(R)$ be a birational projective morphism such that $W$ is normal and $D$ is an effective Cartier divisor on $W$. Then $J=\Gamma(W,\mathcal O_W(-D))\subset R$ is integrally closed.
\end{enumerate}
\end{Lemma}

\begin{proof} We first prove 1). Write $\pi=\beta\alpha$ where $\alpha:Y\rightarrow X$. Now $\beta$ is separated since it is proper by Theorem II.4.9 \cite{H}, so $\alpha$ is projective by Proposition 5.5.5 (v) \cite{EGAII} or Exercise II.4.8 (e) \cite{H}.
We have that $\mathcal O_X\subset \alpha_*\mathcal O_Y\subset K$ and $\alpha_*\mathcal O_Y$ is a sheaf of $\mathcal O_X$-algebras which is a  coherent $\mathcal O_X$-module by Corollary II.5.20 \cite{H} since $\alpha$ is a projective morphism so $\alpha_*\mathcal O_Y=\mathcal O_X$ since $X$ is normal.
Thus
$$
\alpha_*(I\mathcal O_Y)\cong (I\mathcal O_X)\otimes_{\mathcal O_X}\alpha_*\mathcal O_Y\cong I\mathcal O_X
$$
by the projection formula (Exercise II.5.1 (d) \cite{H}). Thus 
$$
\Gamma(Y,I\mathcal O_Y)=\Gamma(X,I\mathcal O_X)=\overline I
$$
by Definition 9.6.2 \cite{LA} and Proposition 9.6.6 \cite{LA}.

2) is Proposition 9.6.11 \cite{LA}.
\end{proof}

\section{Rees algebras of graded filtrations and their Proj}

Let $S=\oplus_{n\ge 0}S_n$ be an $\NN$-graded ring (which is not assumed to be Noetherian). $S_+$ will denote the irrevelant ideal, $S_+=\oplus_{n>0}S_n$.

We will follow the standard conventions on $\mbox{Proj}(S)$ (c.f. Chapter 2, Section 2 of \cite{H}) so that for $H\in S$ homogeneous, $S_{(H)}$ denotes the ring of elements of degree 0 in the localization $S_H$ and for $P\in \mbox{Proj}(S)$, $S_{(P)}$ denotes the elements of degree 0 in $T^{-1}S$, where $T$ is the multiplicative set of homogeneous elements in $S\setminus P$.

For $d>0$ the $d$-th Veronese algebra of $S$ is $S^{(d)}=\oplus_{n\ge 0}S_{dn}$.

We point out the following basic fact which we will use repeatedly throughout this paper. 

\begin{Proposition}\label{PropVer}(Proposition II.4.7 (i) \cite{EGAII}) There is a natural isomorphism $\mbox{Proj}(S^{(d)})\cong \mbox{Proj}(S)$ for all $d>0$.
\end{Proposition}

%\begin{proof} Suppose that $F\in S_n$. Then $F^{d}\in S^{(d)}_n$ and the natural inclusion $S^{(d)}{(F^d)}\rightarrow S_{(F)}$ is an isomorphism. Further, if $G\in S$ is homogeneous of degree $dn$, then $S^{(d)}_{(G^d)}=S^{(d)}_{(G)}$. Thus the induced isomorphisms 
%$\mbox{Spec}(S_{(F)})\cong \mbox{Spec}(S^{(d)}_{(F^d)})$ patch to give an isomorphism
%$$
%\begin{array}{lll}
%\mbox{Proj}(S)&=&\cup_{F\in S\mbox{ homogeneous }}\mbox{Spec}(S_{(F)})\rightarrow \cup_{F\in S\mbox{ homogeneous }}\mbox{Spec}(S^{(d)}_{(F^d)})\\
%&=& \cup_{G\in S^{(d)}\mbox{ homogeneous }}\mbox{Spec}(S^{(d)}_{(G)})=\mbox{Proj}(S^{(d)}).
%\end{array}
%$$
%\end{proof}

\begin{Lemma}\label{Lemma3}
Suppose that $P_1,P_2\in Z:=\mbox{Proj}(S)$ are distinct homogeneous prime ideals. Then $S_{(P_1)}\ne S_{(P_2)}$.
\end{Lemma}

\begin{proof} When thinking of the prime ideal $P_i$ as a point in $Z$, we will denote it by $\alpha_i$,  so that $\mathcal O_{Z,\alpha_i}=S_{(P_i)}$. By Lemma 1.5.10 \cite{BH}, there exists a nonzero homogeneous element $H\in S_+$ such that $H\not\in P_1$ and $H\not\in P_2$. By Proposition II.2.5 \cite{H},
$$
D_+(H)=\mbox{Proj}(S)\setminus V(H)=\{P\in \mbox{Proj}(S)\mid H\in P\}
$$
is open in $Z$ and we have isomorphisms of locally ringed spaces $(D_+(H), \mathcal O_Z\mid D_+(H))\cong \mbox{Spec}(S_{(H)})$.
Since $\alpha_1,\alpha_2\in D_+(H)$ and $D_+(H)$ is affine, we have that $\mathcal O_{Z,\alpha_1}\ne \mathcal O_{Z,\alpha_2}$.
\end{proof}

Let $R$ be a local ring and $\mathcal I=\{I_n\}$ be a graded filtration of ideals in $R$. The Rees Algebra of $\mathcal I$ is
$$
R[\mathcal I]=\oplus_{n\ge 0}I_n.
$$
We will sometimes find it convenient to express $R[\mathcal I]$ as 
$$
R[\mathcal I]=\sum_{n\ge 0}I_nt^n\subset R[t]
$$ 
where $t$ is an indeterminate.

 \section{Valuations, Separatedness and Properness} Separated and proper morphisms are defined and discussed in Section 4 of Chapter II of \cite{H}.

 Let $v$ be a valuation of a field $K$. Associated to $v$ are its valuation ring
 $$
 \mathcal O_v=\{x\in K\mid v(x)\ge 0\},
 $$
  which is a possibly non Noetherian local ring with maximal ideal 
  $$
  m_v=\{x\in K\mid v(x)>0\}.
  $$
  We will follow the relatively standard conventions that 
   $vK$ denotes the value group $\{v(x)\mid x\in K\setminus\{0\}\}$ and $Kv$ denotes the residue field $Kv=\mathcal O_v/m_v$.

 We will say that $v$ dominates a local ring $(A,m_A)$ contained in $K$ if the valuation ring $\mathcal O_v$ dominates $A$; that is, 
 $A\subset \mathcal O_v$ and $m_v\cap \mathcal A=m_A$. If $X$ is an  integral scheme such that $K=\mathcal O_{X,\eta}$ where $\eta$ is the generic point of $X$, and $v$ is a valuation of $K$, we will say that a point $\alpha\in X$ is a center of $v$ on $X$ if $v$ dominates $\mathcal O_{X,\alpha}$.

\begin{Theorem}\label{Theorem1}(Theorem VI.4.5, page 12 \cite{ZS2}) Let $A$ be an integral domain contained in a field $L$ and let $P$ be a prime ideal of $A$. Then there exists a valuation $v$ of $L$ such that  such that $A\subset \mathcal O_v$ and $m_v\cap A=P$.
\end{Theorem}  

As a consequence of Theorem \ref{Theorem1}, we see that if $X$ is an integral $R$-scheme which is birational to $\mbox{Spec}(R)$, and $q\in X$ is a point, then there exists a valuation $v$ of $K$ such that $q$ is a center of $v$ on $X$.

The following lemma is the local analogue of a well known characterization of separatedness and properness for varieties over a field (c.f. \cite{Z1} and Exercise II.4.5 \cite{H}).

\begin{Lemma}\label{LemmaSepProp} Let $R$ be a domain with quotient field $K$.
Suppose that $f:X\rightarrow Y$ is a morphism of integral Noetherian  $R$-schemes which are birational to $\mbox{Spec}(R)$. Then 
\begin{enumerate}
\item[1)] $f$ is separated if and only if whenever  $v$ is a valuation of $K$ which dominates $\mathcal O_{Y,\beta}$ for some  $\beta \in Y$, there exists at most one point $\alpha\in X$ such that $v$ dominates $\mathcal O_{X,\alpha}$.
\item[2)]  $f$ is proper if and only if whenever  $v$ is a valuation of $K$ which dominates $\mathcal O_{Y,\beta}$ for some  $\beta \in Y$,  there is a unique point $\alpha\in X$ such that $v$ dominates $\mathcal O_{X,\alpha}$.
\end{enumerate}
\end{Lemma}

\begin{proof} The only if directions of 1) and 2) are immediate by the respective criteria of 
 Theorem II.4.3 \cite{H} and Theorem II.4.7 \cite{H}. It remains to prove the if direction, so we assume that the valuative criteria of 1) or 2) hold. 
 We will verify that the respective criterion of Theorem II.4.3 \cite{H} or Theorem II.4.7 \cite{H} holds. 
 Suppose that we are given a valuation $\mu$ whose valuation ring $\mathcal O_{\mu}$ has a quotient field $L$ and a commutative diagram of morphisms
\begin{equation}\label{SepProp}
\begin{array}{ccc}
U=\mbox{Spec}(L)&\rightarrow &X\\
\downarrow&&\downarrow f\\
T=\mbox{Spec}(\mathcal O_{\mu})&\rightarrow &Y.
\end{array}
\end{equation}
Let $y_1\in Y$ be the image of the zero ideal of $L$ in $Y$, and let $y_0$ be the image of the maximal ideal of $T$ in $Y$, so that  $\mathcal O_{\mu}$ dominates $\mathcal O_{\overline{\{y_1\}},y_0}$, and the induced map on quotient fields is the inclusion $\kappa(y_1)\rightarrow L$, where $\kappa(y_1)$ is the field $\mathcal O_{Y,y_1}/m_{y_1}$. The morphism $\mbox{Spec}(L)\rightarrow X$ is determined by a point $x_1\in X$,  such that $f(x_1)=y_1$, and an inclusion of fields $\kappa(x_1)\rightarrow L$.
A lifting $T\rightarrow X$ is then equivalent to giving a point $x_0\in X$ such that $f(x_0)=y_0$, $x_0\in \overline{\{x_1\}}$ and $\mathcal O_{\mu}$ dominates $\mathcal O_{\overline{\{x_1\}},x_0}$.  

Let $A=\mathcal O_{Y,y_0}$ and $P$ be the kernel of the map $A\rightarrow \mathcal O_{\mu}$. If $P=0$ then $y_1$ is the generic point of $Y$, and $x_1$ is the generic point of $X$ and so 
the restriction of $\mu$ to $K$ is a valuation of $K$. The restriction $\mu|K$ has at most one center on $X$ if and only if there exists at most one morphism $T\rightarrow X$ making the diagram (\ref{SepProp}) commute, and $\mu|K$ has a unique center on $X$ if and only if there exists a unique morphism $T\rightarrow X$ making the diagram (\ref{SepProp}) commute. 
We may thus assume that $P\ne 0$. Then $\overline{\{x_1\}}\ne X$ and so 
there  exists a divisorial valuation $v$ of $K$ such that $\mathcal O_v$ dominates $\mathcal O_{X,x_1}$. Let $F$ be the free join of $L$ and $Kv$ over $(A/P)_P=\kappa(y_1)$ (this field $F$ exists as is explained in pages 187 - 188 of \cite{ZS1}). It is explained in Section II.7.12 \cite{EGAII} that there exists a valuation $v'$ which extends $v$ which has the residue field $F$. Let $\mu'$ be an extension of $\mu$ to $F$ and $\omega'$ be the composite valuation $\omega'=v'\circ \mu'$.
Composite valuations and their construction are explained on page 43 \cite{ZS2} and on pages 56 - 57 \cite{RTM}. Let $\omega$ be the restriction of $\omega'$ to $K$. 
Suppose that $\omega$ has a center $x_0$ on $X$. Then $x_0$ is a center of $\omega'$ on $X$. By our construction of $\omega'$, $x_1$ is the center of $v'$ on $X$ and $\mathcal O_{\overline{\{x_1\}},x_0}$ is dominated by $\mu'$. Since $\kappa(x_1)\subset L$ and $\mu'|L=\mu$, we have that $\mathcal O_{\mu}$ dominates  $\mathcal O_{\overline{\{x_1\}},x_0}$. Further, if there exists a point $x_0\in X$ such that $x_0\in \overline{\{x_1\}}$ and $\mathcal O_{\mu}$ dominates $\mathcal O_{\overline{\{x_1\}},x_0}$, then $\omega'$ and hence $\omega$ dominates $\mathcal O_{X,x_0}$.

By the above analysis, we see that  $\omega$ has at most one center on $X$ if and only if there exists at most one morphism $T\rightarrow X$ making the diagram (\ref{SepProp}) commute, and 
$\omega$ has a unique center on $X$ if and only if there is a unique morphism $T\rightarrow X$ which makes the diagram (\ref{SepProp}) commute. Thus Lemma \ref{LemmaSepProp} holds by the criteria of Theorem II.4.3 \cite{H} and Theorem II.4.7 \cite{H}.
\end{proof}

\section{Divisorial Filtrations}\label{SecDiv}
 Let $R$ be a  local domain of dimension $d$ with quotient field $K$ and maximal ideal $m_R$.  Let $\nu$ be a discrete valuation of $K$ with valuation ring $\mathcal O_{\nu}$ and maximal ideal $m_{\nu}$.  Suppose that $R\subset \mathcal O_{\nu}$. Then for $n\in \NN$, define valuation ideals
$$
I(\nu)_n=\{f\in R\mid \nu(f)\ge n\}=m_{\nu}^n\cap R.
$$
%We can extend this definition to $\alpha\in \RR$, defining 
%$$
%I(\nu)_{\alpha}=\{f\in R\mid \nu(f)\ge \alpha\}=I(\nu)_{\lceil\alpha\rceil}.
%$$
 
  A divisorial valuation of $R$ (\cite[Definition 9.3.1]{HS}) is a valuation $\nu$ of $K$ such that if $\mathcal O_{\nu}$ is the valuation ring of $\nu$ with maximal ideal $m_{\nu}$, then $R\subset \mathcal O_{\nu}$ and if $p=m_{\nu}\cap R$ then $\mbox{trdeg}_{\kappa(p)}\kappa(\nu)={\rm ht}(p)-1$, where $\kappa(p)$ is the residue field of $R_{p}$ and $\kappa(\nu)$ is the residue field of $\mathcal O_{\nu}$.   
  %If $\nu$ is divisorial valuation of $R$ such that $m_R= m_{\nu}\cap R$, then $\nu$ is called an $m_R$-valuation.
   
 By \cite[Theorem 9.3.2]{HS}, the valuation ring of every divisorial valuation $\nu$ is Noetherian, hence is a  discrete valuation. 
 Suppose that  $R$ is an excellent local domain. Then a valuation $\nu$ of the quotient field $K$ of $R$ which is nonnegative on $R$ is a divisorial valuation of $R$ if and only if the valuation ring $\mathcal O_{\nu}$  of $\nu$ is essentially of finite type over $R$ (\cite[Lemma 5.1]{CPS1}).
 
 In general, the filtration $\mathcal I(\nu)=\{I(\nu)_n\}$ is not Noetherian; that is, the graded $R$-algebra $\sum_{n\ge 0}I(\nu)_nt^n$ is not a finitely generated $R$-algebra. 

An integral  divisorial filtration of $R$ (which we will refer to as a divisorial filtration in this paper) is a filtration $\mathcal I=\{I_m\}$ such that  there exist divisorial valuations $\nu_1,\ldots,\nu_s$ and $a_1,\ldots,a_s\in \ZZ_{\ge 0}$ such that for all $m\in \NN$,
$$
I_m=I(\nu_1)_{ma_1}\cap\cdots\cap I(\nu_s)_{ma_s}.
$$
%A divisorial filtration is called integral (rational) if $a_i\in  \ZZ_{\ge 0}$ for all $i$ ($a_i\in \QQ_{\ge 0}$ for all $i$).

%An integral $s$-divisorial filtration of $R$ (which we will refer to as an $s$-divisorial filtration in this paper) is a filtration $\mathcal I=\{I_m\}$ such that  there exist $s$-valuations $\nu_1,\ldots,\nu_r$ and $a_1,\ldots,a_r\in \ZZ_{\ge 0}$ such that for all $m\in \NN$,
%\begin{equation}\label{eqDF}
%I_m=I(\nu_1)_{ma_1}\cap\cdots\cap I(\nu_r)_{ma_r}.
%\end{equation}

 %Observe that the trivial filtration $\mathcal I=\{I_m\}$, defined by $I_m=R$ for all $m$, is a degenerate case of a divisorial filtration and is a degenerate case of an $s$-divisorial filtration for all $s$.
%The  nontrivial $0$-divisorial filtrations are the  divisorial $m_R$-filtrations of \cite{C3}.

 %We will often denote a divisorial filtration $\mathcal I$ on a local domain $R$ by $\mathcal I=\mathcal I(D)$, even when $R$ is not excellent and there does not exist a representation of  $\mathcal I$ as defined before Theorem \ref{Theorem9}.

$\mathcal I$ is called an $\RR$-divisorial filtration if $a_1,\ldots,a_s\in \RR_{>0}$ and $\mathcal I$ is called a $\QQ$-divisorial filtration if $a_1,\ldots,a_s\in \QQ$. If $a_i\in \RR_{>0}$, then we define
$$
I(\nu_i)_{na_i}:=\{f\in R\mid \nu_i(f)\ge na_i\}=I(\nu_i)_{\lceil na_i\rceil},
$$
where $\lceil x\rceil$ is the round up of a real number.

%Given an ideal $I$ in $R$, the filtration $\{\overline{I^n}\}$ is an example of a divisorial filtration of $R$. The filtration 
%$\{\overline{I^n}\}$ is Noetherian if $R$ is universally Nagata.

 \subsection{Divisorial filtrations on normal excellent local rings} Let $R$ be a normal excellent local ring. Let $\mathcal I=\{I_m\}$ where 
 $$
I_m=I(\nu_1)_{ma_1}\cap\cdots\cap I(\nu_s)_{ma_s}.
$$ 
 for some divisorial valuations $\nu_1,\ldots,\nu_s$ on $R$ be an $\RR$-divisorial filtration on a normal excellent local ring $R$, with $a_1,\ldots, a_s\in \RR_{>0}$. Then there exists a projective birational morphism $\phi:X\rightarrow \mbox{Spec}(R)$ such that $X$ is a normal integral $R$-scheme such that there exist prime divisors $F_1,
 \ldots, F_s$ on $X$ such that $\mathcal O_{\nu_i}=\mathcal O_{X,F_i}$ for  $1\le i\le s$. After possibly replacing $\nu_i$ with an equivalent valuation, $\nu_i$ is the valuation $\nu_i(x)=n$ if $x\in \mathcal O_{X_i,F_i}$ and $x=t_i^nu$ where $t_i$ is regular parameter in $\mathcal O_{X,F_i}$ and $u$ is a unit in $\mathcal O_{X_i,F_i}$.
  Let $D=a_1F_1+\cdots+a_sF_s$, an effective $\RR$-divisor, where $\lceil x\rceil$ is the smallest integer greater than or equal to a real number $x$. Define $\lceil D\rceil=\lceil a_1\rceil F_1+\cdots+\lceil a_s\rceil F_s$, an integral divisor. 
 We have coherent sheaves $\mathcal O_X(-\lceil n D\rceil)$ on $X$ such that 
 \begin{equation}\label{N1}
 \Gamma(X,\mathcal O_X(-\lceil nD\rceil ))=I_n
 \end{equation}
 for $n\in \NN$. If $X$ is nonsingular then $\mathcal O_X(-\lceil nD\rceil)$ is invertible. The formula (\ref{N1}) is independent of choice of $X$. Further, even on a particular $X$, there are generally many different choices of effective $\RR$-divisors $G$ on $X$ such that $\Gamma(X,\mathcal O_X(-\lceil nG\rceil))=I_n$ for all $n\in \NN$. Any choice of a divisor $G$ on such an $X$ for which the formula $\Gamma(X,\mathcal O_X(-\lceil nG\rceil))=I_n$ for all $n\in \NN$ holds will be called a representation of the filtration  $\mathcal I$. 
 
 Given an  $\RR$-divisor $D=a_1F_1+\cdots+a_sF_s$ on $X$ we have a divisorial filtration
 $\mathcal I(D)=\{I(D)_n\}$ where 
 $$
 I(D)_n=\Gamma(X,\mathcal O_X(-\lceil nD\rceil ))=I(\nu_1)_{\lceil na_1\rceil}\cap\cdots\cap I(\nu_s)_{\lceil na_s\rceil} 
 =I(\nu_1)_{ma_1}\cap\cdots\cap I(\nu_s)_{ma_s}.
 $$ 
 %We write $R[D]=R[\mathcal I(D)]$.
 
 \begin{Lemma}\label{LemmaTau} Suppose that $\mathcal I=\{I_n\}$ is a $\QQ$-divisorial filtration of a two dimensional normal excellent local ring $R$ and $\tau:\mbox{Proj}(R[\mathcal I])\rightarrow \mbox{Spec}(R)$ is the natural morphism. Then 
$$
\mbox{Proj}(R[\mathcal I])\setminus \tau^{-1}(m_R)\cong \mbox{Spec}(R)\setminus\{m_R\}.
$$
\end{Lemma} 

\begin{proof} There exist divisorial valuations $v_1,\ldots,v_r$ of $K$ which are nonnegative on $R$ and positive rational numbers $a_1,\ldots,a_r$ such that
$I_n=I(v_1)_{a_1n}\cap\cdots\cap I(v_r)_{a_rn}$ for $n\ge 0$. We can assume that the $v_i$ are pairwise inequivalent valuations. Let $n_0$ be a positive integer such that $a_in_0$ is a positive integer for all $i$. By Proposition \ref{PropVer}, $\mbox{Proj}(R[\mathcal I]^{(n_0)})\cong \mbox{Proj}(R[\mathcal I])$. So we may replace the $a_i$ with $a_in_0$ for all $i$, and so assume that $a_i$ are positive integers for all $i$. Let $\{P_j\}$ be the set of distinct prime ideals which occur as centers of the $v_i$ on $R$, and are not equal to $m_R$.

Suppose that $v_i$ has the center $P_j$ on $R$. Since $R$ is normal and $P_j$ has height one in $R$, $R_{P_j}$ is a one dimensional regular local ring, and thus is a discrete rank 1 valuation ring. We have that $\mathcal O_{v_i}$ dominates $R_{P_j}$. Thus $\mathcal O_{v_i}$ is a localization of $R_{P_j}$ by Theorem VI.3.3, page 8  \cite{ZS2}. 
Thus $R_{P_j}=\mathcal O_{v_i}$ since $\dim R_{P_j}=1$, so $v_i$ is equivalent to the $P_jR_{P_j}$-adic valuation of $R_{P_j}$. 

We may thus assume, after possibly reindexing the $v_i$, that $v_i$ is the $P_iR_{P_i}$-adic valuation for $1\le i\le s$ and $v_i$ has center $m_R$ on $R$ for $s<i\le r$. Let $u_i\in R_{P_i}$ be a regular parameter. Then $u_i=0$ is a local equation of the curve $V(P_i)$ in an open neighborhood of the point $P_i$ in $\mbox{Spec}(R)$. Thus there exist $g_i\in R$ such that $u_i\in R_{g_i}$, $g_i\not\in P_i$ and $P_iR_{g_i}=u_iR_{g_i}$. Write 
$\cap_{j\ne i} P_j=(f_{i,1},\ldots,f_{i,t_i})$. $P_i\in \mbox{Spec}(R)\setminus V(\cap_{j\ne i}P_j)=\cup_l\mbox{Spec}(R_{f_{i,l}})$, so there exists $l$ such that $P_i\in \mbox{Spec}(R_{f_i,l})$. Let $h_i=g_if_{i,l}$. $P_i\in \mbox{Spec}(R_{h_i})$, $P_j\not\in \mbox{Spec}(R_{h_i})$ for $j\ne i$  and $u_iR_{h_i}=P_iR_{h_i}$.
Write $\cap_jP_j=(a_1,\ldots,a_m)$. $\mbox{Spec}(R)\setminus V(\cap_jP_j)=\cup_l\mbox{Spec}(R_{a_l})$. We  have that
$$
\{\mbox{Spec}(R_{a_1}),\ldots,\mbox{Spec}(R_{a_m}), \mbox{Spec}(R_{h_1}),\ldots,\mbox{Spec}(R_{h_s})\}
$$
 is an affine cover of $\mbox{Spec}(R)\setminus\{m_R\}$. For $1\le l\le m$ and $n\ge 1$, $I_nR_{a_l}=R_{a_l}$, so 
 $$
 \tau^{-1}(\mbox{Spec}(R_{a_l}))=\mbox{Proj}(R_{a_l}[I_nR_{a_l}t^n\mid n\ge 1])=\mbox{Proj}(R_{a_l}[t])=\mbox{Spec}(R_{a_l})
 $$
 for $1\le l\le m$. We also have that $I_nR_{h_i}=u_i^{na_i}R_{h_i}$ for $1\le i\le s$ and $n\ge 1$, so that 
 $$
 \tau^{-1}(\mbox{Spec}(R_{h_i}))=\mbox{Proj}(R_{h_i}[I_nR_{h_i}t^n\mid n\ge 1])=\mbox{Proj}(R_{h_i}[u_i^{a_i}t])=\mbox{Spec}(R_{h_i}).
 $$
Thus $\mbox{Proj}(R[\mathcal I])\setminus \tau^{-1}(m_R)\cong \mbox{Spec}(R)\setminus\{m_R\}$.
\end{proof}

\section{Riemann-Roch theorems for curves}\label{SecRR}

We summarize the famous Riemann-Roch theorems for curves. The following theorems are standard over algebraically closed fields. A reference where they are proven over an arbitrary field $k$ is \cite[Section 7.3]{Lin}. The results that we need are stated in \cite[Remark 7.3.33]{Lin}.

Let $E$ be an integral regular projective curve over a field $k$. For $\mathcal F$ a coherent sheaf on $E$ define
$h^i(\mathcal F)=\dim_kH^i(E,\mathcal F)$.

Let $D=\sum a_ip_i$ be a divisor on $E$, where $p_i$ are prime divisors on $E$ (closed points) and $a_i\in \ZZ$. We have an associated invertible sheaf $\mathcal O_X(D)$. Define
 \begin{equation}\label{eq4*}
 \deg(D)=\deg(\mathcal O_E(D))=\sum a_i[\mathcal O_{E_i,p_i}/m_{p_i}:k].
 \end{equation}
The Riemann-Roch formula is
\begin{equation}\label{eq44}
\chi(\mathcal O_{E}(D)):=h^0(\mathcal O_{E}(D))-h^1(\mathcal O_{E}(D))=\mbox{deg}(D)+1-p_a(E)
\end{equation}
where $p_a(E)$ is the arithmetic genus of $E$. 

We further have Serre duality,
\begin{equation}\label{eq42}
H^1(E,\mathcal O_{E}(D))\cong H^0(E,\mathcal O_{E}(K-D))
\end{equation}
where $K=K_E$ is a canonical divisor on $E$. As a consequence, we have 
\begin{equation}\label{eq43}
\deg D>2p_a(E)-2=\deg(K)\mbox{ implies }H^1(E,\mathcal O_{E}(D))=0.
\end{equation}

If $\mathcal L$ is an invertible sheaf on $E$ then $\mathcal L\cong \mathcal O_E(D)$ for some divisor $D$ on $E$, and we may define $\deg(\mathcal L)=\deg(\mathcal O_X(D))=\deg(D)$.  

We will apply the above formulas in the case that $E$ is a prime exceptional divisor for a resolution of singularities $\pi:X\rightarrow\mbox{Spec}(R)$ as in Section \ref{SecRes}. We take $k=R/m_R$. We have that  $E$ is projective over $k=R/m_R$, and $E$ is a nonsingular (by assumption) integral curve.  Let $D$ be a divisor on $X$. Then
$\deg(\mathcal O_X(D)\otimes\mathcal O_E)=(D\cdot E)$.

\section{Initial results on a resolution of singularities}\label{SecRes} This section summarizes a construction and some results in \cite{CNg}.

 Let $(R,m_R)$ be a 2-dimensional normal excellent local ring with quotient field $K$ and $\mathcal I=\{I_n\}$ be a $\QQ$-divisorial filtration of $R$. There exists (by \cite{Li} or \cite{CJS}) a birational projective morphism $\pi:X\rightarrow\mbox{Spec}(R)$ such that $X$ is nonsingular, and an effective  $\QQ$-divisor $D$ on $X$ such that  $I_n=\Gamma(X,\mathcal O_X(-\lceil nD\rceil ))$ for all $n\ge 0$. By Lemma 4.1 \cite{CNg}, there exists a Zariski decomposition of $D$; that is, there exist unique effective $\QQ$-divisors $\Delta$ and $B$ on $X$ such that $\Delta=D+B$, $-\Delta$ is nef ($(-\Delta\cdot E)\ge 0$ if $E$ is an exceptional curve of $\pi$), $B$ has exceptional support and $(\Delta\cdot E)=0$ if $E$ is a component  of $B$.
Then
$$
\Gamma(X,\mathcal O_X(-\lceil n\Delta\rceil))=\Gamma(X,\mathcal O_X(-\lceil nD \rceil ))=I_n
$$
 for $n\ge 0$ by Lemma 4.3 \cite{CNg}. This local Zariski decomposition is a local analogue of the decomposition defined for effective divisors on a nonsingular projective surface by Zariski in \cite{Z} (a more recent reference is \cite{BCK}).
By Lemma 2.5 \cite{CNg}, $\pi:X\rightarrow \mbox{Spec}(R)$ is the blowup of an $m_R$-primary ideal, so 
\begin{equation}\label{eq10}
\pi:X\setminus \pi^{-1}(m_R)\rightarrow \mbox{Spec}(R)\setminus\{m_R\}\mbox{ is an isomorphism.}
\end{equation}

Let $E_1,\ldots,E_r$ be the prime exceptional divisors of $X\rightarrow \mbox{Spec}(R)$.  After possibly blowing up $X$ some more, 
we may assume that $E_1+\cdots+E_r$ is a simple normal crossings divisor on $X$; that is, all $E_i$ are nonsingular, each closed point $p\in E_i$ is contained in at most one other $E_j$, and there exist regular parameters $x,y$ in $\mathcal O_{X,p}$ such that $x=0$ is a local equation of $E_i$ at $p$ and if $p\in E_j$, then $y=0$ is a local equation of $E_j$ at $p$.

Let 
 $$
 R[\mathcal I]=\oplus_{n\ge 0}I_n=\oplus_{n\ge 0}\Gamma(X,\mathcal O_X(-\lceil n\Delta\rceil)),
  $$
  the Rees algebra associated to $\mathcal I$, as defined earlier. Define $Z(\mathcal I)=\mbox{Proj}(R[\mathcal I])$, with natural morphism $\tau:Z(\mathcal I)\rightarrow \mbox{Spec}(R)$. We have that $Z(I)\setminus \tau^{-1}(m_R)\cong \mbox{Spec}(R)\setminus \{m_R\}$ by Lemma \ref{LemmaTau}. 
  
  We have the remarkable fact that $m_RR[\mathcal I]$ has only finitely many minimal primes.

\begin{Proposition}\label{Prop1}(Proposition 6.7 \cite{CNg})
Let 
$$
P_j=P_{E_j}=\oplus_{n\ge 0}\Gamma(X,\mathcal O_X(-\lceil n\Delta\rceil-E_j)\mbox{ for }1\le j\le r.
$$
 Then 
\begin{enumerate}
\item[1)] $P_j$ is a prime ideal in $R[\mathcal I]$ for $1\le j<r$.
\item[2)] $\sqrt{m_RR[\mathcal I]}=\cap_{i=1}^rP_i$.
\end{enumerate}
\end{Proposition}

\begin{Proposition}\label{Prop2}(Proposition 6.9 \cite{CNg})
Let $P_j$ be as defined in Proposition \ref{Prop1}. Then
\begin{enumerate}
\item[1)] $\dim R[\mathcal I]/P_j=2$ if $(\Delta\cdot E_j)<0$.
\item[2)] $\dim R[\mathcal I]/P_j\le 1$ if $(\Delta\cdot E_j)=0$.
\end{enumerate}
\end{Proposition}

We observe that $P_j\not\in \mbox{Proj}(R[\mathcal I])$ if $\dim R[\mathcal I]/P_j=0$, in which case, $P_j=q\oplus R[\mathcal I]_+$ where $q=\pi(E_j)\in \mbox{Spec}(R)$.

\begin{Lemma}\label{Lemma30} $R[\mathcal I]$ is normal.
\end{Lemma}
This follows from Lemma 5.8 of \cite{C5}.

\section{Stable Base Locus}\label{SecsBL} From now on, we will use the notation established in Section \ref{SecRes}.

The concept of the stable base locus of a line bundle on a projective variety is discussed in \cite{LA}. We develop it in our local context. 

Let $\mathcal L$ be an invertible sheaf on $X$. Define the base locus of $\Gamma(X,\mathcal L)$ to be 
$$
\mbox{BL}(\Gamma(X,\mathcal L))=\{p\in X\mid \Gamma(X,\mathcal L)\mathcal O_{X,p}\ne \mathcal L_p\}.
$$
This is a closed subset of $X$ (regarded as an algebraic set and not as a scheme). The stable base locus of $\mathcal L$ is
$$
\mbox{sBL}(\mathcal L)=\cap_{n\ge 1}\mbox{BL}(\Gamma(X,\mathcal L^n)),
$$
which is also a closed subset of $X$. 

The support $\mbox{Supp}(D)$ of a divisor $D=\sum a_iF_i$, where $F_i$ are prime divisors and $a_i\in \ZZ$, is $\mbox{Supp}(D)=\cup_iF_i$.

\begin{Lemma}\label{Lemma4} Let $\mathcal L$ be an invertible sheaf on $X$. Then there exists $n_1$ such that 
$$
\mbox{BL}(\Gamma(X,\mathcal L^{nn_1}))=\mbox{sBL}(\mathcal L)
$$ for all $n\ge n_1$.
\end{Lemma}

The proof of Lemma \ref{Lemma4} is essentially the same as the  the proof of Proposition 2.1.21 of \cite{LA}, using the fact that $I_mI_n\subset I_{m+n}$ for all $m,n$.

 We will denote $\mbox{sBL}(\mathcal O_X(-n\Delta))$ by $\mbox{sBL}(-n\Delta)$ when $n\Delta$ is an integral divisor. Recall that $-\Delta$ is nef. 

\begin{Lemma}\label{Lemma5} Let $n_0$ be the smallest positive integer $n$ such that $n\Delta$ is an integral divisor and $BL(\Gamma(X,\mathcal O_X(-nn_0\Delta)))=\mbox{sBL}(-n_0\Delta)$ for all $n>0$.
Then $\mbox{sBL}(-n_0\Delta)$ is a finite union of exceptional curves $E$ for $\pi$, all of which satisfy $(E\cdot \Delta)=0$.
\end{Lemma}

\begin{proof} After replacing $\Delta$ by   $n_0\Delta$ we may suppose that $\Delta$ is an integral divisor and $\mbox{BL}(\Gamma(X,\mathcal O_X(-n\Delta))=\mbox{sBL}(-\Delta)$ for all $n>0$.

Suppose  that there exists $q\in \mbox{sBL}(-n_0\Delta)$ which is an isolated closed point. We will derive a contradiction. Then $q$ is isolated in $\mbox{BL}(\Gamma(X,\mathcal O_X(-\Delta)))$. Write $-\Delta=-M+B$ where $M$ and $B$ are effective divisors such that $\mbox{Supp}(B)=\mbox{sBL}(-\Delta)$,
$\Gamma(X,\mathcal O_X(-M))=\Gamma(X,\mathcal O_X(-\Delta))$ and $\Gamma(X,\mathcal O_X(-M))\mathcal O_X=\mathcal K\mathcal O_X(-M)$ where $\mathcal K$ is an ideal sheaf of $S$ such that $\mathcal O_X/\mathcal K$ has finite support. We necessarily have that $q$ is in the support of $\mathcal O_X/\mathcal K$ but 
$q\not\in \mbox{Supp}(B)$. By Proposition 5.2 \cite{CNg}, there exists $n>0$ such that $\mathcal O_X(-nM)$ is generated by global sections, so in particular,
$\Gamma(X,\mathcal O_X(-nM))\mathcal O_{X,q}=\mathcal O_X(-nM)_q$. Since $q\not\in \mbox{Supp}(B)$, we have 
$$
\begin{array}{lll}
\mathcal O_X(-n\Delta)_q&=&\mathcal O_X(-nM)_q=\Gamma(X,\mathcal O_X(-nM))\mathcal O_{X,q}\\
&\subset& \Gamma(X,\mathcal O_X(-nM+nB))\mathcal O_{X,q}=\Gamma(X,\mathcal O_X(-n\Delta))\mathcal O_{X,q}.
\end{array}
$$
Thus $q\not\in \mbox{BL}(\Gamma(X,\mathcal O_X(-n\Delta))$, a contradiction. We thus have that $\mbox{sBL}(-\Delta)$ is a finite union of curves. Now by Proposition 5.1 \cite{CNg}, $\mbox{sBL}(-\Delta)$ is a union of exceptional curves $E$ for $\pi$ which must satisfy $(\Delta\cdot E)=0$. 

\end{proof}

%Analogues of Lemma \ref{Lemma11} are well known in related situations. 

\begin{Lemma}\label{Lemma11} Let $n_0$ be the smallest positive integer $n$ such that $n\Delta$ is an integral divisor. Then $\mbox{sBL}(-n_0\Delta)=\emptyset$ if and only if $R[\mathcal I]$ is a finitely generated $R$-algebra.
\end{Lemma}

\begin{proof} Suppose that $\mbox{sBL}(-n_0\Delta)=\emptyset$. Then, by Lemma \ref{Lemma4}, there exists $n_1>0$ such that $\mbox{BL}(\Gamma(X,\mathcal O_X(-nn_0n_1\Delta)))=\emptyset$ for $n\ge 1$. Thus $\Gamma(X,\mathcal O_X(-n_0n_1\Delta))$ determines an $R$-morphism
$
\psi_{n_0n_1}:X\rightarrow Y=\mbox{Proj}(T)
$
where $T$ is the graded algebra $R=R[t\Gamma(X,\mathcal O_X(-n_0n_1\Delta))]$. 
Since $X$ is the blowup of an ideal $J$ in $R$, $X$ is the blowup of the ideal sheaf $J\mathcal O_Y$ in $Y$, so  $\psi_{n_0n_1}$ is a projective morphism.
Let $\mathcal O_Y(1)$ be the very ample divisor on $Y$ determined by the sheafification of the graded $T$-module $T(1)$. Then $\phi_{n_0n_1}^*(\mathcal O_Y(n))=\mathcal O_X(-nn_0n_1\Delta)$ for $n\in \ZZ$. Now 
$$
\Gamma(X,\mathcal O_X(-n\Delta))=\Gamma(Y,(\phi_{n_0n_1})_*\mathcal O_X(-n\Delta))
$$
for all $n\ge 0$ and 
$$
(\phi_{n_0n_1})_*\mathcal O_X(-n\Delta)\cong ((\phi_{n_0n_1})_*\mathcal O_X(-r\Delta))\otimes\mathcal O_Y(m)
$$
where $n=mn_0n_1+r$ with $0\le r<n_0n_1$ by the projection formula (c.f. Exercise II.5.1 (d) of \cite{H}). Further, the sheaves $(\phi_{n_0n_1})_*\mathcal O_X(-r\Delta)$ are coherent $\mathcal O_Y$-modules by Theorem III.8.8 \cite{H} since $\psi_{n_0n_1}$ is a projective morphism.

Now 
$A:=\oplus_{n\ge 0}\Gamma(Y,\mathcal O_Y(n))$ is a finitely generated $R[t\Gamma(X,\mathcal O_X(-n_0n_1\Delta))]$ module, and hence is a finitely generated $R$-algebra by Theorem 11.47 \cite{AG} (or as established in the course of  the proof of Theorem II.5.19 \cite{H}). 

We will now show that if $m\in \ZZ$, then $\oplus_{n\ge 0}\Gamma(Y,\mathcal O_Y(m+n))$ is a finitely generated $A$-module. Suppose that $m\ge 0$. Then $\oplus_{n\ge 0}\Gamma(Y,\mathcal O_Y(m+n))\subset A$  is an ideal which is thus a finitely generated $A$-module. Further, 
$$
\oplus_{n\ge 0}\Gamma(Y,\mathcal O_Y(n-m))
=\Gamma(Y,\mathcal O_Y(-m))\oplus \cdots\oplus\Gamma(Y,\mathcal O_Y(-1))\oplus A
$$
is a finitely generated $A$-module since $\Gamma(Y,\mathcal O_Y(i))$ is a finitely generated $R$-module for all $i$ (by Theorem III.8.8 \cite{H} or by Theorem II.5.2 \cite{H}). 

More generally, suppose that $\mathcal F$ is a coherent $\mathcal O_Y$-module. We will show that $\oplus_{n\ge 0}\Gamma(Y,\mathcal F(n))$ is a finitely generated $A$-module. To prove this, first observe that by Corollary II.5.18 \cite{H}, there exists a short exact sequence of coherent sheaves on $Y$
$$
0\rightarrow\mathcal K\rightarrow \oplus_{i=1}^s\mathcal O_Y(m_i)\rightarrow\mathcal F\rightarrow 0
$$ 
for some $m_i\in \ZZ$,
giving an exact sequence
$$
\oplus_{i=1}^s\left(\oplus_{n\ge 0}\Gamma(Y,\mathcal O_Y(n+m_i))\right)\rightarrow \oplus_{n\ge 0}\Gamma(Y,\mathcal F(n))
\rightarrow \oplus_{n\ge 0}H^1(Y,\mathcal K(n)).
$$
Now $H^1(Y,\mathcal K(n))$ are finitely generated $R$-modules and $H^1(Y,\mathcal K(n))=0$ for $n\gg 0$ by Theorem II.5.2 \cite{H}. Thus $\oplus_{n\ge 0}\Gamma(Y,\mathcal F(n))$ is a finitely generated $A$-module. 

In particular,
$$
\oplus_{n\ge 0}\Gamma(X,\mathcal O_X(-n\Delta))=\oplus_{r=0}^{n_0n_1-1}\left(\oplus_{n\ge 0} \Gamma(Y,(\psi_{n_0n_1})_*\mathcal O_X(-r\Delta)\otimes\mathcal O_Y(n))\right)
$$
is a finitely generated $A$-module and hence a finitely generated $R$-algebra since $(\psi_{n_0n_1})_*\mathcal O_X(-r\Delta)$ are coherent $\mathcal O_Y$-modules, so $R[\mathcal I]$ is a finitely generated $R$-algebra. 

Now suppose that $R[\mathcal I]$ is a finitely generated $R$-algebra, but $\mbox{sBL}(-n_0\Delta)\ne\emptyset$. We will derive a contradiction. There exists a positive integer $n_1$ such that $\mbox{BL}(\Gamma(X,\mathcal O_X(-nn_0n_1\Delta)))=\mbox{sBL}(-n_0n_1\Delta)$ for all $n\ge 1$ and the Veronese algebra $R[\mathcal I]^{(n_0n_1)}$ is generated in degree 1.
With our assumptions, there exists an exceptional curve $E$ of $\pi$ such that $E\subset \mbox{sBL}(-n_0\Delta)$ by Lemma \ref{Lemma5},  so that
\begin{equation}\label{eq60}
\Gamma(X,\mathcal O_X(-nn_0n_1\Delta))\mathcal O_{X,E}=\Gamma(X,\mathcal O_X(-n_0n_1\Delta))^n\mathcal O_{X,E}
\end{equation}
for all $n\ge 1$.  Since $E\subset \mbox{sBL}(-n_0\Delta)$, we have that  $E\subset \mbox{BL}(\Gamma(X,\mathcal O_X(-nn_0)))$ for all $n>0$. Thus 
$$
\Gamma(X,\mathcal O_X(-n_0n_1\Delta))\mathcal O_{X,E}=t^a\mathcal O_X(-n_0n_1\Delta)_E
$$
for some $a>0$ where $t$ is a generator of the maximal ideal of $\mathcal O_{X,E}$. Thus
$$
\Gamma(X,\mathcal O_X(-nn_0n_1\Delta)\mathcal O_{X,E}=t^{na}\mathcal O_X(-nn_0n_1\Delta)\mathcal O_{X,E}
$$
for all $n\ge 1$ by (\ref{eq60}). Now since $-n_0\Delta$ is nef, by Proposition 5.3 \cite{CNg}, there exists $b>0$ such that 
$$
\Gamma(X,\mathcal O_X(-nn_0n_1\Delta))\mathcal O_{X,E}=t^{b_n}\mathcal O_X(-nn_0n_1\Delta)\mathcal O_{X,E}
$$
where $b_n\le b$ for all $n>0$. Thus $0<na\le b$ for all $n>0$, which is impossible. Thus $\mbox{sBL}(-n_0\Delta)=\emptyset$.
\end{proof}

\begin{Lemma}\label{Lemma2} Let $n_0$ be the smallest positive integer $n$ such that $n\Delta$ is an integral divisor. Suppose that $F$ is an exceptional curve for $\pi$, $F\not\subset \mbox{sBL}(-n_0\Delta)$ and $F\cap\mbox{sBL}(-n_0\Delta)\ne\emptyset$. Then $(\Delta\cdot F)<0$.
\end{Lemma}
\begin{proof} After replacing $\Delta$ with $n_0\Delta$, we may assume that $n_0=1$.
Assume that $(\Delta\cdot F)=0$. We will derive a contradiction. Since $F\not\subset \mbox{sBL}(-\Delta)$, there exists $n$ such that the image of the natural map $\sigma:\Gamma(X,\mathcal O_X(-n\Delta))\rightarrow\Gamma(F,\mathcal O_X(-n\Delta)\otimes\mathcal O_F)$ is nonzero, so there exists $s\in \Gamma(X,\mathcal O_X(-n\Delta))$ such that $\sigma(s)\ne 0$. Let $\mathcal L=\mathcal O_X(-n\Delta)\otimes_{\mathcal O_X}\mathcal O_F$. %$\mathcal L\cong \mathcal O_X(\sum a_ip_i)$ where $p_i\in F$ are closed points and
We have that $\deg(\mathcal L)=(-n\Delta\cdot F)=0$.
The divisor $(\sigma(s))$ of $\sigma(s)$ on $F$ is the effective divisor on $F$ such that $\sigma(s)\mathcal O_F=\mathcal O_F(-(\sigma(s))\mathcal L$. Thus $\mbox{deg}(\sigma(s))=\deg\mathcal L=0$, and so
 $(\sigma(s))$ is the zero divisor since $(\sigma(s))$ is an effective divisor. 
Thus $\sigma(s)$ does not vanish on $F$ and so $F\cap\mbox{sBL}(-\Delta)=\emptyset$, a contradication.
\end{proof}

\section{Relationship with analytic spread}\label{SecAS}

After replacing $\Delta$ with a multiple of $\Delta$ we may assume that $\Delta$ is an integral divisor and $\mbox{BL}(\Gamma(X,\mathcal O_X(-n\Delta)))=\mbox{sBL}(-\Delta)$ for all $n\ge 1$.  We use the notation of Proposition \ref{Prop1}, associating to an exceptional divisor $E$ of $\pi$ the prime ideal $P_E=\oplus_{n\ge 0}\Gamma(X,\mathcal O_X(-n\Delta-E))$ of $R[\mathcal I]$.

\begin{Proposition}\label{Prop5} Suppose that $\mbox{sBL}(-\Delta)\ne\emptyset$ and $F$ is an exceptional curve of $\pi$ such that $F$ is not contained in $\mbox{sBL}(-\Delta)$. Then there exists a prime exceptional divisor $G$ of $X$ (possibly equal to $F$) such that $P_G\subset P_F$ and $\dim R[\mathcal I]/P_G=2$.
\end{Proposition}

\begin{proof}  If $(F\cdot \Delta)<0$ then $\dim R[\mathcal I]/P_F=2$ by  Proposition \ref{Prop2}.  In particular, by Lemma \ref{Lemma2}, we may assume that $F$ is disjoint from $\mbox{sBL}(-\Delta)$. 
 Suppose that $(F\cdot\Delta)=0$.
Since $\pi^{-1}(m_R)$ is connected (by Corollary III.11.4 \cite{H}), there exist exceptional curves $E_0,\ldots,E_r$ of $\pi$ which are not contained in $\mbox{sBL}(-\Delta)$ such that $F=E_r$, $E_i\cap \mbox{sBL}(-\Delta)=\emptyset$ for $1\le i\le r$, $E_0\cap\mbox{SBL}(-\Delta)\ne\emptyset$  and $E_i\cap E_{i-1}\ne\emptyset$ for $1\le i\le r$.
Let $q_i\in E_i\cap E_{i-1}$ for $1\le i\le r$. Let $s$ be the largest  integer with $0\le s< r$ such that $(E_s\cdot\Delta)<0$ ($s$ exists by Lemma \ref{Lemma2}). Then $(E_i\cdot\Delta)=0$ and the $E_i$ are disjoint from $\mbox{sBL}(-\Delta)$ for $s<i\le r$. Thus the point $q_i\not\in \mbox{sBL}(-\Delta)$ for $s<i\le r$, and so $\Gamma(X,\mathcal O_X(-n\Delta)\otimes\mathcal I_{q_i})\ne\Gamma(X,\mathcal O_X(-n\Delta))$ for $s<i\le r$ and $n>0$  where $\mathcal I_{q_i}$ is the ideal sheaf of the point $q_i$ in $\mathcal O_X$.
Let $Q_i:=\oplus_{n\ge 0}\Gamma(X,\mathcal O_X(-n\Delta)\otimes\mathcal I_{q_i})$, $P_{i}=P_{E_i}=\oplus_{n\ge 0}\Gamma(X,\mathcal O_X(-n\Delta-E_{i}))$. $P_i\subset Q_{i}$ are homogeneous prime ideals in $R[\mathcal I]$ which are in $Z(\mathcal I)=\mbox{Proj}(R[\mathcal I])$ and contain $\sqrt{m_RR[\mathcal I]}$. Thus $\dim R[\mathcal I]/Q_i>0$.
Now $(E_{i}\cdot\Delta)=0$ for $s<i\le r$ so $\dim S/P_{i}\le 1$ for $s<i\le r$ by Proposition \ref{Prop2}, and so $\dim R[\mathcal I]/Q_i=1$ for $s<i\le r$ and  $Q_i= P_{i}$ for $s<i\le r$.  We thus have a chain of inclusions
$$
P_r=Q_{r}\supset P_{r-1}=Q_{r-1}\supset \cdots \supset P_{s+1}=Q_{s+1}.
$$
Now $P_r\supset Q_{s+1}$ and $1=\dim R/P_r=\dim R/Q_{s+1}$ so $P_r=Q_{s+1}\supset P_s$ and $\dim R[\mathcal I]/P_s=2$ by Proposition \ref{Prop2}.
\end{proof}

%The analytic spread $\ell(\mathcal I)$ of $\mathcal I$ is $\ell(\mathcal I)=\dim R[\mathcal I]/m_RR[\mathcal I]$. We have that  $0\le \ell(\mathcal I)\le \dim R=2$ by Lemma 3.6  \cite{CPS}.

%\begin{Theorem}\label{Theorem2}  The following occurs in the three possible cases $0,1,2$ of analytic spread $\ell(\mathcal I)$ of $\mathcal I$.
%\begin{enumerate} 
%\item[1)] If $\ell(\mathcal I)=0$, then $\sqrt{m_RR[\mathcal I]}=m_R\oplus R[\mathcal I]_+$ and $R[\mathcal I]$ is not a finitely generated $R[\mathcal I]$-algebra.
%\item[2)] If $\ell(\mathcal I)=1$, then $R[\mathcal I]$ is a finitely generated $R$-algebra and there is a unique exceptional curve $E$ such that $\sqrt{m_RR[\mathcal I]}=P_E$  with $\dim R[\mathcal I]/P_E=1$. 
%\item[3)] If $\ell(\mathcal I)=2$, then $\sqrt{m_RR[\mathcal I]}=\cap P_{E_i}$ where the intersection is over all the prime ideals $P_{E_i}$ such that 
%$\dim R[\mathcal I]/P_{E_i}=2$. 
%\end{enumerate}
%\end{Theorem}

\subsection{Proof of Theorem \ref{Theorem2}}

%\begin{proof}

Suppose that $\ell(\mathcal I)=0$. Then $\tau^{-1}(m_R)=\emptyset$. If $R[\mathcal I]$ is a finitely generated $R$-algebra, then $\tau:Z(\mathcal I)\rightarrow \mbox{Spec}(R)$ is projective and hence proper, so it is surjective and $\pi^{-1}(m_R)\ne\emptyset$, giving a contradiction. Thus $R[\mathcal I]$ is not a finitely generated $R$-algebra. By Proposition \ref{Prop1}, $P_E\not\in Z(\mathcal I)$ for all exceptional curves $E$ of $\pi$. Since $P_E\cap R=\{f\in R\mid v_E(f)>0\}$ and $E$ contracts to the maximal ideal $m_R$ of $R$, we have that $P_E=m_R\oplus R[\mathcal I]_+$ for all exceptional curves $E$ of $\pi$ and so  $\sqrt{m_RR[\mathcal I]}=m_R\oplus R[\mathcal I]_+$.

Suppose that $\ell(\mathcal I)=1$. We will show that then $\mbox{sBL}(-\Delta)=\emptyset$. Suppose that $\mbox{sBL}(-\Delta)\ne \emptyset$. By Lemma \ref{Lemma5}, there exists an exceptional curve $F$ which is contained in $\mbox{sBL}(-\Delta)$. Since $\ell(\mathcal I)=1$, by Proposition \ref{Prop1} and Proposition \ref{Prop2} we have that $(E\cdot\Delta)=0$ for all exceptional curves $E$ of $\pi$, so  since $\pi^{-1}(m_R)$ is connected by Corollary III.11.4 \cite{H} and     by Lemma \ref{Lemma2}, $\pi^{-1}(m_R)\subset \mbox{sBL}(-\Delta)$. 
Thus for all exceptional curves $E$ of $\pi$, $\Gamma(X,\mathcal O_X(-n\Delta-E))=\Gamma(X,\mathcal O_X(-n\Delta))=I_n$ for all $n\ge 1$ and so 
$P_E\not\in Z(\mathcal I)$, so $\ell(\mathcal I)=0$ by Proposition \ref{Prop1}, giving a contradiction. Thus  $\ell(\mathcal I)=1$ implies $\mbox{sBL}(-\Delta)=\emptyset$. By Lemma \ref{Lemma11}, $R[\mathcal I]$ is a finitely generated $R$-algebra and so $\tau:Z(\mathcal I)\rightarrow \mbox{Spec}(R)$ is a projective birational morphism. 
Since $R$ is normal, $\tau^{-1}(m_R)$ is connected by Corollary III.11.4 \cite{H}. But $\tau^{-1}(m_R)$ is a finite union of closed points since
$\dim \tau^{-1}(m_R)=\ell(\mathcal I)-1=0$. Thus by Proposition \ref{Prop1}, $\sqrt{m_RR[\mathcal I]}$ is a prime ideal $P_E$ in $R[\mathcal I]$ such that $\dim R[\mathcal I]/P_E=1$.

Suppose that $\ell(\mathcal I)=2$. First suppose that $\mbox{sBL}(-\Delta)=\emptyset$. Then $R[\mathcal I]$ is a finitely generated $R$-algebra by Lemma \ref{Lemma11}, so that $\tau$ is a projective birational morphism. Since $R$ is normal, $\tau^{-1}(m_R)$ is connected by Corollary III.11.4 \cite{H}. Since $\dim \tau^{-1}(m_R)=\ell(\mathcal I)-1=1$, $\tau^{-1}(m_R)$ is a (connected) union of curves. Now suppose that $\mbox{sBL}(-\Delta)\ne \emptyset$. If $E$ is an exceptional curve of $\pi$ and $E\subset \mbox{sBL}(-\Delta)$, then $P_E=m_R\oplus R[\mathcal I]_+$ since $E\subset \mbox{BL}(\Gamma(X,\mathcal O_X(-n\Delta)))$ for all $n\ge 1$. By Proposition \ref{Prop1} and Proposition \ref{Prop5}, $\sqrt{m_RR[\mathcal I]}=\cap P_{E_i}$ where the intersection is over the prime ideals $P_{E_i}$ such that 
$\dim R[\mathcal I]/P_{E_i}=2$. 
%\end{proof}

\begin{Remark}\label{Thm2RM} Distinct exceptional curves $E$ of $\pi$ such that $(E\cdot\Delta)<0$ (which are the exceptional curves $E$ such that $\dim R[\mathcal I]/P_E=2$ by Proposition \ref{Prop2}) have distinct prime ideals $P_E$ in $R[\mathcal I]$. In particular, the prime ideals $P_{E_i}$ such that $\dim R/P_{E_i}=2$ of the intersection $\sqrt{m_RR[\mathcal I]}=\cap P_{E_i}$
in the case $\ell(\mathcal I)=2$ of Theorem \ref{Theorem2} are all distinct.  We could give a direct proof of this now, but since we will not need this fact, and it follows from Theorem \ref{Theorem6}, we will not give a separate proof here, but refer to Theorem \ref{Theorem6} and the commentary after it. 
\end{Remark}

\section{an invariant}\label{SecInv}
Let $v$ a valuation of $K$ which is nonnegative on $R$. The facts that
$I_m^n\subset I_{mn}$ and $v(I_m^n)=nv(I_m)$ imply
$$
\frac{v(I_{mn})}{mn}\le \frac{nv(I_m)}{mn}=\frac{v(I_m)}{m},
$$
and so 
\begin{equation}\label{eq2}
\frac{v(I_{mn})}{mn}\le
\min\{\frac{v(I_m)}{m},\frac{v(I_n)}{n}\}.
\end{equation}
We define
$$
\gamma_v(\mathcal I)=\inf\{\frac{v(I_n)}{n}\}.
$$
where the infimum is over all positive $n$. This invariant is also defined earlier in Subsection \ref{SubIn} of the introduction.

By (\ref{eq2}), if $\frac{v(I_m)}{m}=\gamma_v(\mathcal I)$ for some $m$, then
$$
\frac{v(I_{mn})}{mn}=\gamma_v(\mathcal I)\mbox{ for all $n>0$.}
$$

\begin{Proposition}\label{Prop6} Suppose that $v$ is a valuation of $K$ which is nonnegative on  $R$. Then $v$ has a center on $Z(\mathcal I)$ if and only if there exists a positive integer $m$ such that 
$$
\frac{v(I_m)}{m}=\gamma_v(\mathcal I).
$$
\end{Proposition}

\begin{proof} First suppose that there does not exist $m$ such that $\frac{v(I_m)}{m}=\gamma_v(\mathcal I)$ but $v$ does have a center on $Z(\mathcal I)$. Then there exists a homogeneous prime ideal $P$ in $R[\mathcal I]$ such that $\oplus_{n>0}I_n\not \subset P$ and $R[\mathcal I]_{(P)}$ is dominated by $\mathcal O_v$.  Thus, there exists $F\in I_m$ for some $m$ such that $0\ne Ft^m\not\in P$, and $R[\mathcal I]_{(F)}\subset R[\mathcal I]_{(P)}\subset \mathcal O_v$. By assumption, there exists $n$ such that $
\frac{v(I_n)}{n}<\frac{v(I_m)}{m}$. So $\frac{v(I_{mn})}{mn}<\frac{v(I_m)}{m}$ by (\ref{eq2}). So there exists an  element $H\in I_{mn}$ such that 
$\frac{v(H)}{mn}=\frac{v(I_{mn})}{mn}<\frac{v(I_m)}{m}$. Since $v(F)\ge v(I_m)$, this implies that $v(\frac{H}{F^n})<0$, so that $\frac{H}{F^n}\not\in \mathcal O_v$. But then
$$
\frac{H}{F^n}=\frac{Ht^{mn}}{(Ft^m)^n}\in R[\mathcal I]_{(F)}
$$
is not in $\mathcal O_v$, a contradiction. Thus if there does not exist $m$ such that $\frac{v(I_m)}{m}=\gamma_v(\mathcal I)$ then $v$ doesn't have a center on $Z(\mathcal I)$.

Now suppose that there exists $m_0$ such that $\frac{v(I_{m_0})}{m_0}=\gamma_v(\mathcal I)$. Then by (\ref{eq2}),
$$
\frac{v(I_{nm_0})}{nm_0}\le \frac{v(I_{m_0})}{m_0}=\gamma_v(\mathcal I)
$$
implies
$$
\frac{v(I_{nm_0})}{nm_0}=\gamma_v(\mathcal I)
$$
for $n\ge 1$.
%We may replace $\Delta$ with $m_0\Delta$, so that $m_0=1$. 
There exists $F\in I_{m_0}$ such that $v(F)=v(I_{m_0})=m_0\gamma_v(\mathcal I)$. Suppose that $x\in R[\mathcal I]_{(Ft^{m_0})}$. Then for some $n>0$
and $G\in I_{m_0n}$,
$$
x=\frac{Gt^{nm_0}}{(Ft^{m_0})^n}=\frac{G}{F^n}.
$$
 So
$$
\frac{v(x)}{nm_0}=\frac{v(G)-nv(F)}{nm_0}\ge \frac{v(I_{nm_0})-nv(F)}{nm_0}=0.
$$
Thus $x\in \mathcal O_v$ and so $R[\mathcal I]_{(Ft^{m_0})}\subset \mathcal O_v$. Let $q=m_{v}\cap R[\mathcal I]_{(Ft^{m_0})}$. 
Then $q\in \mbox{Spec}(R[\mathcal I]_{(Ft^{m_0})})\subset Z(\mathcal I)$ and $\mathcal O_{Z(\mathcal I),q}=(R[\mathcal I]_{(Ft^{m_0})})_q$ is a local ring of $Z(\mathcal I)$ which is dominated by $v$.
Thus if there exists $m_0$ such that $\frac{v(I_{m_0})}{m_0}=\gamma_v(\mathcal I)$ then $v$ has a center on $Z(\mathcal I)$.
\end{proof}

\begin{Proposition}\label{Prop8} Let $v$ be a valuation of $K$ which is nonnegative on $R$. Let $n_0\Delta$ be a multiple of $\Delta$ such that 
$n_0\Delta$ is a $\ZZ$-divisor and  
$\mbox{BL}(\Gamma(X,\mathcal O_X(-nn_0\Delta))=\mbox{sBL}(-n_0\Delta)$ for all $n>0$. 
\begin{enumerate}
\item[1)] If the center of $v$ on $X$ is not in $\mbox{sBL}(-n_0\Delta)$, then there exists $n>1$ such that $\frac{v(I_n)}{n}=\gamma_v(\mathcal I)$.
\item[2)] Suppose that $v$ is a divisorial valuation. Then the center of $v$ on $X$ is not in $\mbox{sBL}(-n_0\Delta)$ if and only if there exists $n>1$ such that $\frac{v(I_n)}{n}=\gamma_v(\mathcal I)$.
\end{enumerate}
\end{Proposition}
We will prove the full converse of 1) in Theorem \ref{Theorem7}.

\begin{proof} We first prove 1). We have that $\frac{\gamma_v(\{I_{nn_0}\})}{n_0}=\gamma_v(\mathcal I)$. Further,   $\frac{v(I_m)}{m}=\gamma_v(\mathcal I)$ for some $m$ if and only if $\frac{v(I_{ln_0})}{l}=\gamma_v(\{I_{nn_0}\}$ for some $l$. Thus we may assume that $n_0=1$.

 Let $\Delta=\sum a_iF_i$, where the $F_i$ are prime divisors and $a_i$ are positive integers. 
Suppose that the center $q$ of $v$ on $X$ is not in $\mbox{sBL}(-\Delta)$. Then since $\mbox{BL}(\Gamma(X,\mathcal O_X(-n\Delta)))=\mbox{sBL}(-\Delta)$ for all $n>0$, we have that  
$$
\Gamma(X,\mathcal O_X(-n\Delta))\mathcal O_{X,q}=\mathcal O_X(-\sum na_iF_i)_q=\prod t_i^{na_i}\mathcal O_{X,q}
$$
where $t_i=0$ are local equations of $F_i$ at $q$ (equal to 1 if $q\not\in F_i$). Thus $v(I_n)=n(\sum_{q\in F_i}a_iv(t_i))$ for all $n\ge 1$, so 
$\frac{v(I_n)}{n}=\gamma_v(\mathcal I)$.

We now prove 2). Suppose that $v$ is a divisorial valuation and the center $q$ of $v$ on $X$ is in $\mbox{sBL}(-\Delta)$. Then the center $q$ of $v$ on $X$ is contained in $\pi^{-1}(m_R)$. There exists a birational projective morphism $\epsilon:W\rightarrow X$ such that $W$ is nonsingular and the center of $v$ on $W$ is a nonsingular curve $F$, supported in the preimage of $m_R$, such that $\mathcal O_v=\mathcal O_{W,F}$. Thus $\Gamma(W,\mathcal O_W(-n\epsilon^*(\Delta)))=\Gamma(X,\mathcal O_X(-n\Delta))=I_n$ for all $n$ and $\mbox{sBL}(-\epsilon^*(\Delta))=\epsilon^{-1}(\mbox{sBL}(-\Delta))$. Let $t$ be a generator of the maximal ideal of $\mathcal O_{W,F}$. Write $\epsilon^*(\Delta)=aF+\cdots$.
$F\subset \mbox{sBL}(-\epsilon^*(\Delta))$ implies $I_n\mathcal O_{X,F}=t^{an+c_n}\mathcal O_{X,F}$ where $c_n>0$ for all $n$. Since $-\epsilon^*(\Delta)$ is nef, there exists a bound $b$ such that $c_n<b$ for all $n>0$ by Proposition 5.3  \cite{CNg}. Thus $\gamma_v(\mathcal I)=a$ but $\frac{v(I_n)}{n}>a$ for all $n$.

\end{proof}

\section{$\mbox{Proj}(R[\mathcal I])$  is Noetherian}\label{SecNoeth}  
After replacing $\Delta$ with a multiple of $\Delta$, we can assume that $\Delta$ is an effective  $\ZZ$-divisor and that 
$\mbox{BL}(\Gamma(X,\mathcal O_X(-n\Delta)))=\mbox{sBL}(-\Delta)$ for all $n>0$. We assume this reduction throughout this section.

For $n>0$, let 
$$
\pi_n:Y_n=\mbox{Proj}(R[t^n\Gamma(X,\mathcal O_X(-n\Delta))])\rightarrow \mbox{Spec}(R) 
$$
be the natural $R$-morphism.
$Y_n$ is a projective integral $R$-scheme which is birational to $\mbox{Spec}(R)$.  We have isomorphisms
\begin{equation}
X\setminus \pi^{-1}(m_R)\cong Y_n\setminus \pi_n^{-1}(m_R)\cong \mbox{Spec}(R)\setminus \{m_R\}
\end{equation}
by (\ref{eq10}).
Let $\mathcal J_n$ be the ideal sheaf on $X$ defined by 
$$
\Gamma(X,\mathcal O_X(-n\Delta))\mathcal O_X=\mathcal O_X(-n\Delta)\mathcal  J_n.
$$
We have that the support 
$$
\mbox{Supp }\mathcal O_X/\mathcal J_n = \mbox{sBL}(-\Delta)\subset \pi^{-1}(m_R).
$$
Let $W_n$ be the  normalization of the blowup of $\mathcal J_n$. 
Now $\mathcal J_n\mathcal O_{W_n}=\mathcal O_{W_n}(-A)$ is an invertible sheaf and so $\Gamma(X,\mathcal O_X(-n\Delta))\mathcal O_{W_n}=\mathcal O_{W_n}(-\epsilon_n^*(-n\Delta)-A)$ is invertible.
Thus we have a commutative diagram of birational projective $R$-morphisms
$$
\begin{array}{ccccc}
&\epsilon_n&W_n&\beta_n\\
&\swarrow&&\searrow&\\
X&&&&Y_n\\
&\searrow&&\swarrow\\
&&\mbox{Spec}(R).
\end{array}
$$

Suppose that $E$ is an exceptional curve of $\pi$ and  that $(E\cdot \Delta)<0$. Then $E$ is not contained in $\mbox{sBL}(-\Delta)$ by Lemma \ref{Lemma5} and
$\dim R[\mathcal I]/P_{E}=2$  by Proposition \ref{Prop2}. 
Thus there exists $n_1$ such that  $P_{E}$ contracts to a prime ideal $P_{E,n}$ in $R[t^nI_n]$ such that $\dim R/P_{E,n}=2$ for all exceptional curves $E$ such that $(E\cdot \Delta)<0$ and $n$ is a multiple of $n_1$. 

Since $E\not\subset \mbox{sBL}(-\Delta)$, we have that $P_E=\oplus_{l\ge 0}P_{E,l}$ where 
$$
P_{E,l}=\Gamma(X,\mathcal O_X(-l\Delta-E))=\{x\in I_n\mid v_E(x)>lv_E(I_1)\}
$$
where $v_E$ is the natural valuation associated to $E$. Suppose that $n$ is a multiple of $n_1$. Let $F$ be the image by $\beta_n$ of the strict transform of $E$ on $W_n$ and $P_F$ be the prime ideal in $R[t^nI_n]$  corresponding to $F$. Then $P_F=\oplus_{l\ge 0}(P_F)_{ln},$
where 
$$
(P_F)_{ln}=\{x\in I_n^l\mid v_E(x)>l_E(I_n)\}.
$$
Since $v(I_n)=nv(I_1)$, we have that $P_F\subset P_E\cap R[t^nI_n]=P_{E,n}$ and since $\pi_n(F)=m_R$, $m_RR[t^nI_n]\subset P_F$. 
Thus since 
$$
2=\dim R[t^nI_n]/m_RR[t^nI_n]=\dim R[t^nI_n]/P_{E,n},
$$
we have that $P_F=P_{E,n}$, and so $F$ is a curve on $Y_n$.

After replacing $\Delta$ with $n_1\Delta$, we may assume that for all $n\ge 1$, 
\begin{equation}\label{eq14}
\begin{array}{l}
\mbox{If 
the strict transform of  an exceptional curve $E$ of $\pi$  on $W_n$}\\
\mbox{is contracted to a point on $Y_n$, then $(E_i\cdot \Delta)=0$.}
\end{array}
\end{equation}

We have  natural $R$-morphisms $\phi_n: X\setminus \mbox{sBL}(-\Delta)\rightarrow Y_n$, defined as follows. Write $I_1=(F_1,\ldots,F_r)$ and $I_m=(H_1,\ldots,H_{s(m)})$ where $H_i=F_i^m$ for $1\le i\le r$.

 Suppose that $q\in X\setminus \mbox{sBL}(-\Delta)$. Since $\mbox{sBL}(-\Delta)=\mbox{BL}(\Gamma(X,\mathcal O_X(-\Delta)))$, there exists $F_i$ such that $F_i$ is a generator of 
 $\mathcal O_X(-\Delta)_q$. Thus there exists a neighborhood $U_q$ of $q$ in $X\setminus \mbox{sBL}(-\Delta)$ such that $F_i$ generates $\mathcal O_X(-\Delta)_q$ in $U_q$, and $F_i^m$ generates $\mathcal O_X(-m\Delta)$ on $U_q$. We may thus write $H_j=F_i^mh_j$ for $1\le j\le s(m)$, where $h_j\in\Gamma(U_q,\mathcal O_X)$. Thus we have a natural morphism $U_q\rightarrow \mbox{Spec}(R[\frac{I_m}{F_i^m}])\subset Y_1$, determined by the $R$-algebra homomorphism $R[\frac{I_m}{F_i^m}]\rightarrow \Gamma(U_q,\mathcal O_X)$ defined by mapping $\frac{H_j}{F_i^m}$ to $h_j$. These local morphisms patch to give a morphism $\phi_m:X\setminus \mbox{sBL}(-\Delta)\rightarrow Y_m$, whose image is in the open subset of $Y_m$ which is the union of the affine open subsets $\mbox{Spec}(R[\frac{I_m}{F_i^m}])$ for $1\le i\le r$.  
 The natural injective homomorphisms $R[\frac{I_m}{F_i^m}]\rightarrow R[\frac{I_n}{F_i^n}]$ for $n\ge m$ induce natural morphisms defined on open subsets of $Y_n$ 
 $$
 \alpha_{m,n}: \cup_{i=1}^{r} \mbox{Spec}(R[\frac{I_n}{F_i^n}])\rightarrow \cup_{i=1}^{r} \mbox{Spec}(R[\frac{I_m}{F_i^m}])
  $$
 which satisfy $\alpha_{l,m}\circ\alpha_{m,n}=\alpha_{l,n}$ for $l\le m\le n$. We have commutative diagrams of morphisms
 \begin{equation}\label{eq12}
 \begin{array}{ccccc}
&\phi_n&X\setminus \mbox{sBL}(-\Delta)&\phi_m\\
&\swarrow&&\searrow&\\
   \cup_{i=1}^{r} \mbox{Spec}(R[\frac{I_n}{F_i^n}])        &&\stackrel{\alpha_{m,n}}{\rightarrow}&&\cup_{i=1}^{r} \mbox{Spec}(R[\frac{I_m}{F_i^m}]).\\ 
  \end{array}
  \end{equation}
Let $V_n=\beta_n(\epsilon_n^{-1}(\mbox{sBL}(-\Delta)))$ for $n\ge 1$. Since $W_1$ is normal, $\beta_1$ factors as $W_1\stackrel{\overline \beta_1}{\rightarrow} \overline Y_1\rightarrow Y_1$ where $\overline Y_1$ is the normalization of $Y_1$. Let $\overline V_1=\overline \beta_1(\epsilon_1^{-1}(\mbox{sBL}(-\Delta)))$ in $\overline Y_1$. $W_n$, $Y_n$ and $\overline{Y_1}$  are projective $R$-schemes, and hence are proper $R$-schemes by Theorem II.4.9 \cite{H}. Thus $\beta_n$ and $\overline \beta_1$ are projective morphisms by Proposition II.5.5.5 (v) \cite{EGAII}.   Hence $V_n$ and $\overline V_1$ are  closed in $Y_n$ and $\overline Y_1$ respectively.  Let $W_{1,n}$ be the normalization of the blowup of $\mathcal J_1\mathcal J_n$. 
We have natural morphisms $\gamma_1:W_{1,n}\rightarrow W_1$ and $\gamma_n:W_{1,n}\rightarrow W_n$. 

Let $\epsilon_{1,n}=\epsilon_1\circ \gamma_1:W_{1,n}\rightarrow X$, $\overline \delta_1=\overline\beta_1\circ \gamma_1:W_{1,n}\rightarrow\overline{Y_1}$, $\delta_1=\beta_1\circ \gamma_1:W_{1,n}\rightarrow Y_1$, $\delta_n=\beta_n\circ \gamma_n:W_{1,n}\rightarrow Y_n$. Then
$$
\overline V_1=\overline{\beta_1}(\epsilon_1^{-1}(\mbox{sBL}(-\Delta))=\overline{\delta}_1(\epsilon_{1,n}^{-1}(\mbox{sBL}(-\Delta))),
$$
$$
V_n=\beta_n(\epsilon_n^{-1}(\mbox{sBL}(-\Delta)))=
\delta_n(\epsilon_{1,n}^{-1}(\mbox{sBL}(-\Delta))).
$$

\begin{Proposition}\label{Prop35}
\begin{equation}\label{eq11}
\mbox{There exists $n_0$ such that $\beta_{nn_0}^{-1}(V_{nn_0})=\epsilon_{nn_0}^{-1}(\mbox{sBL}(-\Delta))$ for $n\ge 1$.}
\end{equation}
\end{Proposition}

\begin{proof}
Certainly $\epsilon_n^{-1}(\mbox{sBL}(-\Delta))\subset \beta_n^{-1}(V_n)$ for all $n\ge 1$. Let $q\in X\setminus \mbox{sBL}(-\Delta)$ and  $q_n:=\phi_n(q)\in V_n$ for $n\ge 1$.   By the discussion before equation (\ref{eq12}),
There exists $F_i\in I_1$ such that 
 $q_n\in \mbox{Spec}(R[\frac{I_n}{F_i^n}])$, and $q_1\in \mbox{Spec}(R[\frac{I_1}{F_i}])$.

The ideals $I_n$ are integrally closed for all $n$.  
There exists $n_0$ such that 

\begin{equation}\label{eq13}
\mbox{$(\overline{I_1^{n_0}})^n$ is integrally closed for all $n\ge 1$,}
\end{equation}
since the integral closure of $R[tI_1]$ in $R[t]$ is a finite $R[tI_1]$-algebra because $R$ is  excellent. $I_1^{n_0}\subset I_{n_0}$ and $I_{n_0}$ integrally closed implies $\overline{I_1^{n_0}}\subset I_{n_0}$ which implies that 
\begin{equation}\label{eq15}
(\overline {I_1^{n_0}})^n\subset (I_{n_0})^n\subset I_{n_0n}\mbox{ for all $n\ge 1$}. 
\end{equation}
For $F\in I_1$, we have natural inclusions 
\begin{equation}\label{eq16}
R[\frac{I_1}{F}]=R[\frac{I_1^{n_0}}{F^{n_0}}]\subset R[\frac{\overline{I_1^{n_0}}}{F^{n_0}}]=R[\frac{(\overline{I_1^{n_0}})^n}{F^{n_0n}}]
\subset R[\frac{I_{nn_0}}{F^{n_0n}}].
\end{equation}
 By equation (\ref{eq13}), and since $R$ is normal,
\begin{equation}\label{eq73}
R[\overline{I_1^{n_0}}t^{n_0}]\mbox{ is normal}
\end{equation}
 and thus 
 \begin{equation}\label{eq74}
 R[\frac{\overline{I_1^{n_0}}}{F^{n_0}}]=R[t^{n_0}\overline{I_1^{n_0}}]_{(t^{n_0}F^{n_0})}\mbox{ is integrally closed},
 \end{equation}
 where $R[t^{n_0}\overline{I^{n_0}}]_{(t^{n_0}F^{n_0})}$ is
 the set of homogeneous elements of degree 0 in the localization $R[\overline{I_1^{n_0}}t^{n_0}]_{t^{n_0}F^{n_0}}$. Since $q\in X\setminus\mbox{sBL}(-\Delta)$, we have that
$$
I_1^{n_0}\mathcal O_{X,q}\subset \overline{I_1^{n_0}}\mathcal O_{X,q}\subset I_{n_0}\mathcal O_{X,q}=\mathcal O_X(-n\Delta)_q=I_1^{n_0}\mathcal O_{X,q}.
$$
By the discussion before (\ref{eq12}), we have an $R$-morphism
$$
\overline\phi_1:X\setminus \mbox{sBL}(-\Delta)\rightarrow \cup_{i=1}^r\mbox{Spec}(R[\frac{\overline{I_1^{n_0}}}{F_i^{n_0}}])\subset \overline{Y_1}
$$
which factors $\phi_1$, and by (\ref{eq16}), we have $R$-morphisms
$$
\overline{\alpha}_{nn_0}:\cup_{i=1}^r\mbox{Spec}(R[\frac{I_{nn_0}}{F_i^{nn_0}}])\rightarrow \cup_{i=1}^r\mbox{Spec}(R[\frac{\overline{I_1^{n_0}}}{F_i^{n_0}}])
$$
such that $\overline{\alpha_{nn_0}}\phi_{nn_0}=\overline\phi_1$.
Let $\overline q_1=\overline{\phi}_1(q)$.

Suppose that $q_{nn_0}\in V_{nn_0}$ for some $n\ge 1$; that is,  $\delta_{nn_0}^{-1}(q_{nn_0})$ intersects $\epsilon_{1,nn_0}^{-1}(\mbox{sBL}(-\Delta))$. 

Let $p\in \delta_n^{-1}(q_{nn_0})$, and let $v$ be a valuation of $K$ which is centered at $p$ on $W_{1,n}$. Then $v$ is centered at $q_{nn_0}$ on $Y_{nn_0}$, which implies that $v$ is centered at $\overline q_1=\overline{\alpha}_{nn_0}(q_{nn_0})$ on $\overline{Y_1}$.
Thus $\overline{\delta}_1(p)=\overline{q_1}$, and so $\delta_{nn_0}^{-1}(q_{nn_0})\subset \overline{\delta_1}^{-1}(q_1)$. So
   $\overline \delta_1^{-1}(\overline q_1)$ intersects 
$\epsilon_{1,nn_0}^{-1}(\mbox{sBL}(-\Delta))$. Thus there exists $\tilde q\in \overline \delta_1^{-1}(\overline q_1)\cap \epsilon_{1,nn_0}^{-1}(\mbox{sBL}(-\Delta))$.
Now $X\setminus\mbox{sBL}(-\Delta)\cong W_{1,nn_0}\setminus \epsilon_{1,nn_0}^{-1}(\mbox{sBL}(-\Delta))$, so we may identify $q$ with its preimage in $W_{1,nn_0}$. Since $\mathcal O_{\overline Y_1,\overline q_1}$ is normal, $\overline\delta_1^{-1}(\overline q_1)$ is connected by Corollary III.11.4 \cite{H}. Thus there exists a closed integral curve $A$ in $W_{1,nn_0}$ such that $A$ intersects but is not contained in $\epsilon_{1,nn_0}^{-1}(\mbox{sBL}(-\Delta))$ and $\overline{\delta_1}$ contracts $A$ to the point $\overline q_1$ on $\overline Y_1$. 
Let $B=\epsilon_{1,nn_0}(A)$. Now $B$ is a curve on $X$ since $B$ intersects but is not contained in $\mbox{sBL}(-\Delta)$.
Necessarily $B$ is an exceptional curve of $\pi$, since $\overline{\delta_1}$ is an isomorphism away from the preimage of $m_R$.
  By (\ref{eq14}), we have that $(B\cdot \Delta)=0$. But by Lemma \ref{Lemma2}, this is impossible as $B$ intersects $\mbox{sBL}(-\Delta)$. This contradiction shows that $q_{nn_0}\not\in V_{nn_0}$ for all $n\ge 1$, so $\beta_{nn_0}^{-1}(V_{nn_0})\subset \epsilon_{nn_0}^{-1}(\mbox{sBL}(-\Delta))$
 for all $n\ge 1$ and  we have  thus established  Proposition \ref{Prop35}.

\end{proof}

\begin{Corollary}\label{Cor1} Let $n_0$ be the constant of Proposition \ref{Prop35}. Then after replacing $\Delta$ with $n_0\Delta$ (and $F_i$ with $F_i^{n_0}$ for $1\le i\le r$), we have that $\phi_n:X\setminus\mbox{sBL}(-\Delta)\rightarrow Y_n\setminus V_n$ is a proper morphism for all $n\ge 1$.
\end{Corollary}

\begin{proof} By Proposition \ref{Prop35}, 
$$
\beta_n^{-1}(V_n)=\epsilon_n^{-1}(\mbox{sBL}(-\Delta)) 
$$
for all $n$. Since $\beta_n:W_n\rightarrow Y_n$ is proper, $W_n\setminus \beta_n^{-1}(V_n)\rightarrow Y_n\setminus V_n$ is proper by Corollary II.4.8 (c) \cite{H}. 
Since $X\setminus\mbox{sBL}(-\Delta)\cong W_n\setminus \epsilon_n^{-1}(\mbox{sBL}(-\Delta))$, $\phi_n:X\setminus\mbox{sBL}(-\Delta)\rightarrow Y_n\setminus V_n$  is proper.
\end{proof}

By  equation (\ref{eq12}) and Corollary \ref{Cor1}, 
$$
\alpha_{m,n}(Y_n\setminus V_n)=
\alpha_{m,n}\phi_n(X\setminus \mbox{sBL}(-\Delta))=\phi_m(X\setminus \mbox{sBL}(-\Delta))=Y_m\setminus V_m.
$$
Thus  $\alpha_{m,n}:Y_n\setminus V_n\rightarrow Y_m\setminus V_m$ is surjective.

The morphisms $Y_n\rightarrow \mbox{Spec}(R)$ are projective, hence they are separated, by Theorem II.4.9 \cite{H}. Hence  $Y_n\setminus V_n\rightarrow \mbox{Spec}(R)$ is separated by Corollary II.4.6 (a) and (b) \cite{H}, and thus $\alpha_{m,n}$ is separated by Corollary II.4.6 (e) \cite{H}.

Let $v$ be a valuation of $K$ which has a center on $Y_m\setminus V_m$. Then since $\phi_m$ is proper, there exists $q\in X\setminus \mbox{sBL}(-\Delta)$ which is the center of $v$ on $X\setminus \mbox{sBL}(-\Delta)$, and so $\phi_n(q)$ is the center of $v$ on $Y_n\setminus V_n$. Thus $\alpha_{m,n}$ is proper by Lemma \ref{LemmaSepProp}.

We now show that $\alpha_{m,n}:Y_n\setminus V_n\rightarrow Y_m\setminus V_m$ is quasi-finite. The notion of a quasi-finite morphism is defined in Exercsie II.3.5 \cite{H} and on page 5 of \cite{Mil}.  If $\alpha_{m,n}$ isn't quasi-finite,  then there exists a curve $C$ in $Y_n\setminus V_n$ such that $\alpha_{m,n}(C)=q_m\in Y_m\setminus V_m$ where $q_m$ is a closed point. Since $\phi_n:X\setminus \mbox{sBL}(-\Delta)\rightarrow Y_n$ is proper, there exists a curve $\gamma$ in $X\setminus\mbox{sBL}(-\Delta)$ such that $\phi_n(\gamma)=C$, and so $\phi_m(\gamma)=q_m$. Since $\phi_m$ is an isomorphism on $X\setminus\pi^{-1}(m_R)$, the closure of $\gamma$ in $X$ is an exceptional curve $E$ of $\pi$. By (\ref{eq14}), $(E\cdot \Delta)=0$, and by Lemma \ref{Lemma2}, $E$ is disjoint from $\mbox{sBL}(-\Delta)$. In particular, $\gamma=E$. Since $E$ is disjoint from $\mbox{sBL}(-\Delta)$, we have that $\mathcal O_X(-n\Delta) \otimes \mathcal O_E$ is generated by global sections, and $(\Delta\cdot E)=0$ implies this is a degree zero line bundle on $X$. Thus $\mathcal O_X(-n\Delta)\otimes\mathcal O_E\cong \mathcal O_E$ and so $E$ is contracted to a point by $\phi_n$, a contradiction. Thus $\alpha_{m,n}$ is quasi finite. Since $\alpha_{m,n}$ is proper and quasi-finite, it is finite by Corollary 1.10 \cite{Mil}.  We have proven the following proposition.

\begin{Proposition} After replacing $\Delta$ with a multiple of $\Delta$, for $n>m\ge 1$, there are commutative diagrams of  $R$-morphisms
$$
\begin{array}{ccccc}
&\phi_{n}&X\setminus \mbox{sBL}(-\Delta)&\phi_{m}\\
&\swarrow&&\searrow&\\
Y_{n}\setminus V_{n}&&\stackrel{\alpha_{m,n}}{\rightarrow}&&Y_{m}\setminus V_{m}\\
\end{array}
$$
such that  $\phi_{n},\phi_{m}$ and $\alpha_{m,n}$ are birational and proper and $\alpha_{m,n}$ is finite.
\end{Proposition}

There exist homogeneous forms $G_1t^{d_1},\ldots,G_st^{d_s}\in R[tI_1]$ such that $V_1=Z(G_1t^{d_1},\ldots,G_s t^{d_s})$. Replacing the $G_jt^{d_j}$ with suitable powers, we may assume that  the $d_j$ are all equal to a common $d$; that is, $d_j=d$ for all $j$. After replacing $\Delta$ with $d\Delta$ and $F_i$ with $F_i^d$ for $1\le i\le r$, we may assume that $d=1$.

Let $N$ be the normalization of $Y_1\setminus V_1$.
The morphisms $\alpha_{1,n}:Y_n\setminus V_n\rightarrow Y_1\setminus V_1$ are finite and birational so for $n>m$ we have a factorization
\begin{equation}\label{eq71}
N\rightarrow Y_n\setminus V_n\rightarrow Y_m \setminus V_m\rightarrow Y_1\setminus V_1.
\end{equation}

As in equations (\ref{eq13}) and (\ref{eq73}), there exists $n_0$ such that $R[t^{n_0}\overline{I_1^{n_0}}]$ is normal and as in (\ref{eq15}), $\overline{I_1^{n_0}}\subset I_{n_0}$.
As in equation (\ref{eq16}), for $F\in I_1$, 
$$
R[\frac{I_1}{F}]\subset R[\frac{\overline{I_1^{n_0}}}{F^{n_0}}]\subset R[\frac{I_{n_0}}{F^{n_0}}]
$$
where $R[\frac{\overline{I_1^{n_0}}}{F^{n_0}}]$ is integrally closed by (\ref{eq74}). Thus $R[\frac{\overline{I_1^{n_0}}}{F^{n_0}}]$ is the normalization of $R[\frac{I_1}{F}]$, which is contained in $R[\frac{I_{n}}{F^{n}}]$ for $n\ge n_0$. 

Since $Y_1\setminus V_1\subset \cup_{i=1}^r \mbox{Spec}(R[\frac{I_1}{F_i}])$, we have that $Y_1\setminus V_1=\cup_{i,j}\mbox{Spec}(R[\frac{I_1}{F_i},\frac{F_i}{G_j}])$. 

Let $g_{ij}=\frac{G_j}{F_i}$. We have that
$$
R[\frac{I_1}{F_i},\frac{F_i}{G_j}]=R[\frac{I_1}{F_i}]_{g_{ij}}=R[\frac{I_1^{2}}{F_iG_j}]=R[tI_1]_{(F_iG_jt^{2})},
$$
the elements of degree 0 in the localization $R[tI_1]_{F_iG_jt^{2}}$ and for $n\ge 1$,
\begin{equation}\label{eq41}
R[\frac{I_n}{F_i^n},\frac{F_i}{G_j}]=R[\frac{I_n}{F_i^n}]_{g_{ij}}=R[\frac{I_n^{2}}{F_i^nG_j^n}]=R[t^nI_n]_{(F_i^nG_j^nt^{2n})}.
\end{equation}
The natural morphisms $\psi_n:N\rightarrow Y_n\setminus V_n$ of (\ref{eq71}) are finite, hence affine. We have factorizations $\alpha_{m,n}\psi_n=\psi_m$ for $m<n$. Consider the morphisms
$$
\alpha_{m,n}:\cup_{i=1}^r\mbox{Spec}(R[\frac{I_n}{F_i^n}])\rightarrow \cup_{i=1}^r\mbox{Spec}(R[\frac{I_m}{F_i^m}])
$$
of (\ref{eq12}).  We have that $\alpha_{m,n}^{-1}(\mbox{Spec}(R[\frac{I_m}{F_i^m}])=\mbox{Spec}(R[\frac{I_n}{F_i^n}])$, since 
$$
\mbox{Spec}(R[\frac{I_m}{F_i^m}])\cap \mbox{Spec}(R[\frac{I_m}{F_j^m}])=
\mbox{Spec}(R[\frac{I_m}{F_i^m},\frac{F_i^m}{F_j^m}])=\mbox{Spec}(R[\frac{I_m}{F_j^m},\frac{F_j^m}{F_i^m}]).
$$

Recall that $Y_n\setminus V_n\subset \cup_{i=1}^r\mbox{Spec}(R[\frac{I_n}{F_i^n}])$ for all $n$. We will establish the following formula for all $n\ge 1$:
\begin{equation}\label{eq80}
Y_n\setminus V_n=\alpha_{1,n}^{-1}(Y_1\setminus V_1)
\end{equation}
where $\alpha_{1,n}$ is the morphism of equation (\ref{eq12}), so that $\alpha_{1,n}^{-1}(Y_n\setminus V_n)$ is a subset of $\cup_{i=1}^r\mbox{Spec}(R[\frac{I_n}{F_i^n}])$.

Suppose that $q_n\in \alpha_{1,n}^{-1}(Y_1\setminus V_1)$. Let $v$ be a valuation of $K$ which has center $q_n$ on $Y_n$. Let $p$ be the center of $v$ on $X$ (which exists since $v$ is nonnegative on $R$ and $X$ is projective over $\mbox{Spec}(R))$. The center of $v$ on $\cup_{i=1}^r\mbox{Spec}(R[\frac{I_1}{F_i}])$ is $q_1=\alpha_{1,n}(q_n)$, which is in $Y_1\setminus V_1$ by assumption. But $\phi_1:X\setminus \mbox{sBL}(-\Delta)\rightarrow Y_1\setminus V_1$ is proper by Corollary \ref{Cor1}, so the center $p$ of $v$ on $X$ is in $X\setminus \mbox{sBL}(-\Delta)$. Thus $q_n=\phi_n(p)\in Y_n\setminus V_n$, and so $\alpha_{1,n}^{-1}(Y_1\setminus V_1)\subset Y_n\setminus V_n$. Now suppose that $q_n\in Y_n\setminus V_n$. Then $q_n=\phi_n(p)$ for some $p\in X\setminus \mbox{sBL}(-\Delta)$ by Corollary \ref{Cor1}, which implies that $\alpha_{1,n}(q_n)=\phi_1(p)\in Y_1\setminus V_1$. Thus $Y_n\setminus V_n\subset \alpha_{1,n}^{-1}(Y_1\setminus V_1)$, establishing (\ref{eq80}).

 %Since $\mbox{Spec}(R[\frac{I_m}{F_i^m}]_{g_{ij}})$ is the open set $\mbox{Spec}(R[\frac{I_m}{F_i^m}])\setminus V(g_{ij})$ of $\mbox{Spec}(R[\frac{I_m}{F_i^m}])$, we have that
%$$
%\alpha_{1,n}^{-1}(\mbox{Spec}(R[\frac{I_1}{F_i}]_{g_{ij}})=\mbox{Spec}(R[\frac{I_{n}}{F_i}]_{g_{ij}})
%$$ 
%

It now follows from (\ref{eq80}) that
$$
Y_n\setminus V_n=\cup_{i,j}\alpha_{1,n}^{-1}(\mbox{Spec}(R[\frac{I_1}{F_i}]_{g_{ij}})=\cup_{i,j}\mbox{Spec}(R[\frac{I_n}{F_i^n}]_{g_{ij}})
$$
for all $n\ge 1$.

We  have that 
$$
\psi_n^{-1}(\mbox{Spec}(R[\frac{I_{n}}{F_i^n}]_{g_{ij}})=\mbox{Spec}(A_{ij})
$$
where $A_{ij}$ is the integral closure of $R[\frac{I_1}{F_i}]_{g_{ij}}$ in $K$ for $1\le i\le r$, $1\le j\le s$ and $n\ge 1$, 

with natural finite birational inclusions
$$
R[\frac{I_1}{F_i}]_{g_{ij}}\rightarrow R[\frac{\overline{I_1^{n_0}}}{F_i^{n_0}}]_{g_{ij}}\rightarrow
R[\frac{I_{n_0}}{F_i^{n_0}}]_{g_{ij}}\rightarrow R[\frac{I_{n}}{F_i^n}]_{g_{ij}}\rightarrow A_{ij}
$$
for $n\ge n_0$.

Since $R[\frac{\overline{I_1^{n_0}}}{F_i^{n_0}}]_{g_{i,j}}$ is integrally closed it is equal $A_{ij}$. Thus 
\begin{equation}\label{eq61}
\mbox{$R[\frac{I_{n}}{F_i^n}]_{g_{ij}}=A_{ij}$ is integrally closed for $n\ge n_0$,}
\end{equation}
and so 
\begin{equation}\label{eq81}
Y_n\setminus V_n=N\mbox{ for }n\ge n_0.
\end{equation}

We may now establish the following theorem.

\begin{Theorem}\label{Theorem4} There is a natural open immersion $\Psi:N\rightarrow Z(\mathcal I)=\mbox{Proj}(R[\mathcal I])$.
The image of $\Psi$ is 
$$
\{P\in \mbox{Proj}(R[\mathcal I])\mid \mbox{ there exists a valuation $v$  of $K$ whose center is on $X\setminus\mbox{sBL}(-\Delta)$}\}
$$
\end{Theorem}

\begin{proof} By equations (\ref{eq41}) and (\ref{eq61}), 
%and since 
%$$
%R[\frac{I_{n(d+1)}}{F_i^nG_j^n}]\subset R[\frac{I_{n(d+1)}}{F_i^{n(d+1)}},\frac{F_i^d}{G_j}],
%$$
we have that
$$
R[\mathcal I]_{(F_iG_jt^{2})}=\cup_{n\ge 0}R[t^nI_n]_{(F_i^nG_j^nt^{2n})}=A_{i,j}.
$$
Thus there are natural isomorphisms $\mbox{Spec}(A_{ij})\rightarrow \mbox{Spec}(R[\mathcal I]_{(F_iG_j)})$ which are compatible with restriction, so we have an open immersion of 
$N=\cup_{i,j}\mbox{Spec}(A_{ij})$ into $\mbox{Proj}(R[\mathcal I])$.

Suppose that a  valuation $v$ of $K$  has a center $q\in X\setminus \mbox{sBL}(-\Delta)$. Then $v$ has center $q_n=\phi_n(q)\in Y_n\setminus V_n$ for all $n$. For $n\gg 0$, we have  shown that $\mathcal O_{Y_n,q_n}$ is a local ring of $Z(\mathcal I)$ so that $v$ has a center on $Z(\mathcal I)$. 

Suppose that $v$ has a center in the image of $\Psi$. Then $v$ dominates $\mathcal O_{Z(\mathcal I),\alpha}$ for some point $\alpha$ in the image of $\Psi$, and so $\mathcal O_{Z(\mathcal I),\alpha}\cong \mathcal O_{Y_n,q_n}$ for some $q_n\in Y_n\setminus V_n$ with $n\gg 0$. But $\mathcal O_{Y_n,q_n}$ is dominated by $v$ and $\phi_n$ is proper, so $v$ has a center on $X\setminus \mbox{sBL}(-\Delta)$.

\end{proof}

\begin{Lemma}\label{Lemma22} Suppose that there exists a valuation $v$ of $K$  such that $v$ has center on $X$ in $\mbox{sBL}(-\Delta)$
and has center $\alpha\in Z(\mathcal I)$. If $\omega$ is another valuation of $K$ which has center at $\alpha$ on $Z(\mathcal I)$ then $\omega$ has center in $\mbox{sBL}(-\Delta)$ on $X$.
\end{Lemma}

\begin{proof} Suppose that the center of $\omega$ on $X$ is 
 in $X\setminus \mbox{sBL}(-\Delta)$ and the center of $\omega$  on $Z(\mathcal I)$ is $\alpha$. We will derive a contradiction. 
By Theorem \ref{Theorem4}, $\mathcal O_{Z,\alpha}=\mathcal O_{Y_n,q_n}$ for some $n\ge 1$, with $q_n=\phi_n(q) \in Y_n\setminus V_n$.
Since the center of $v$ on $Z(\mathcal I)$ is $\alpha$, the center of $v$ on $Y_n$ is $q_n$. But $q_n\in Y_n\setminus V_n$, so the center of $v$ on $X$ is in $X\setminus\mbox{sBL}(-\Delta)$  since $\phi_n$ is proper, giving a contradiction.

\end{proof}

\begin{Theorem}\label{Theorem5} Suppose that $v$ is a valuation of $K$ which has a center in $\mbox{sBL}(-\Delta)$ on $X$. Then $v$ does not have a center on $Z(\mathcal I)$.
\end{Theorem}

\begin{proof} 
After replacing $\Delta$ with a multiple of $\Delta$, we may assume that $\Delta$ is an integral divisor and $\mbox{BL}(\Gamma(X,\mathcal O_X(-n\Delta)))=\mbox{sBL}(-\Delta)$ for all $n\ge 1$.
Suppose that $v$ has the center $\alpha$ on $Z(\mathcal I)$ and the center $q$ in $\mbox{sBL}(-\Delta)$ on $X$. We will denote 
 the prime ideal in $R[\mathcal I]$ corresponding to $\alpha$ by $Q$. Since $\mbox{sBL}(-\Delta)\subset \pi^{-1}(m_R)$, the center of $v$ on $\mbox{Spec}(R)$ is $m_R$ and so $\alpha\in \tau^{-1}(m_R)$. Thus $m_RR[\mathcal I]\subset  Q$.

 Suppose that $\ell(\mathcal I)=0$. Then $\tau^{-1}(m_R)=\emptyset$, contradicting the fact that $\alpha\in \tau^{-1}(m_R)$, so this case cannot occur.

   Suppose that $\ell(\mathcal I)=1$. Then by Theorem \ref{Theorem2}, $R[\mathcal I]$ is a finitely generated $R$-algebra. So  $\mbox{sBL}(-\Delta)=\emptyset$ by Lemma \ref{Lemma11}. Thus  this case also cannot occur.
 
 Now suppose that $\ell(\mathcal I)=2$. Then $m_RR[\mathcal I]=\cap P_{E_i}$ where the $E_i$ are the exceptional curves of $\pi$ which  satisfy $(E_i\cdot\Delta)<0$ by Theorem \ref{Theorem2} and Proposition \ref{Prop2}.
 Since $m_RR[\mathcal I]\subset  Q$, there exists $E=E_i$ such that  $P_E\subset Q$ (and $(E\cdot \Delta)<0$).
Let $A=\mathcal O_{Z(\mathcal I),\alpha}=R[\mathcal I]_{(Q)}$. Then
$$
(0)\subset \tilde P:=(P_E)_{(Q)}\subset m_{A}=(Q)_{(Q)}.
$$
Now $E\not\subset \mbox{sBL}(-\Delta)$ by Lemma \ref{Lemma5} since $(E\cdot\Delta)<0$. Now $\dim R[\mathcal I]/P_E=2$ by Proposition \ref{Prop2}.  After replacing $\Delta$ with a multiple of $\Delta$, we may assume that the constant $n_0$ of 
Proposition \ref{Prop35} is $n_0=1$.
For $n\gg0$, $P_{E,n}:=P_E\cap R[t^{n}I_{n}]$ is such that 
$\dim R[t^{n}I_{n}]/P_{E,n}=2$. Let $D$ be the one dimensional subscheme $\mbox{Proj}(R[\mathcal I]/P_F)$ of 
$Z(\mathcal I)$  and Let $G_n$ be the curve $\mbox{Proj}(R[t^nI_n]/P_{E,n})$
in $Y_{n}=\mbox{Proj}(R[t^{n}I_{n}])$ and $H_n=G_n\cap (Y_{n}\setminus V_{n})$. For $n\gg 0$, $Y_{n}\setminus V_{n}$ is normal by (\ref{eq81}). Since $\phi_{n}$ is birational with $\phi_{n}(E)=H_n$ and $X$, $Y_n\setminus V_n$ are normal,  we have that $\mathcal O_{Y_{n},H}\cong\mathcal O_{X,E}$ for $n\gg 0$. By (\ref{eq81}), $\mathcal O_{Y_{n,H}}\cong \mathcal O_{Z(\mathcal I),D}$ for $n\gg 0$. 
 Thus $A_{\tilde P}\cong \mathcal O_{X,E}$.

 If $Q=P_E$, then the natural valuation $v_E$ of the valuation ring $\mathcal O_{X,E}$ has center $\alpha$ on $Z(\mathcal I)$ but has center $E$ on $X$, which is not contained in $\mbox{sBL}(-\Delta)$, giving a contradiction by Lemma \ref{Lemma22}. Thus we have that $\tilde P$ is not equal to $m_A$.
 
By Theorem \ref{Theorem1}, there exists a valuation $\mu$ of the field $(A/\tilde P)_{\tilde P}$ such that $\mu$ dominates $A/\tilde P$. We have that $R/m_R\subset A/\tilde P$ since $P_E\cap R=m_R$. Thus $(A/\tilde P)_{\tilde P}$ is the function field $L=k(E)$ of the projective curve $E$ over $k=R/m_R$. So $L$ has transcendence degree 1 over $k$ and $\mu$ is trivial on $k\setminus \{0\}$, since $k\subset (A/\tilde P)/m_A(A/\tilde P)\subset L\mu$, the residue field of $\mu$.
Thus $\mu$ is a valuation of the function field $L/k$, so  $\mu L\cong\ZZ$ (e.g. by Theorem VI.14.31 \cite{ZS2}).

Choosing $x\in L$ such that $\mu(x)=1$, we have that  $L$ is a finite extension of $k(x)$ and $\mathcal O_{\mu|k(x)}=k[x]_{(x)}$ (the localization of 
$k[x]$ at the maximal ideal $(x)$ of $k[x]$). Let $n$ be the maximal ideal of the integral closure $B$ of $k[x]_{(x)}$ in $L$ such that $B_n=\mathcal O_{\mu}$ (Proposition 2.38 \cite{RTM}). $B_n$ is a 1 dimensional normal local ring, so it is a regular local ring and $nB_n=(\overline u)$ for some regular parameter $\overline u\in B_n$. Thus $\mathcal O_{\mu}$ is the local ring $\mathcal O_{E,\beta}$ of some closed point $\beta\in E$ (since $E$ is a projective curve over $k$ and so is proper over $k$). 

Let $v_E$ be the divisorial valuation of $K$ corresponding to $E$, so that $\mathcal O_{v_E}=\mathcal O_{X,E}$. Let $\omega=v_E\circ\mu$ be the composite valuation of $v_E$ and $\mu$. This  concept is explained on page 43 \cite{ZS2} and on pages 56-57 of \cite{RTM}. 
%Let 
%$$
%\delta:\mathcal O_{v_E}=\mathcal O_{X,E}\rightarrow \mathcal O_{v_E}/m_{v_E}=\mathcal O_{X,E}/m_E
%$$
%be the residue map. Let $G<\omega L$ be the convex subgroup (isolated subgroup) 
%$$
%G=\{\omega(x)\mid x\in m_{v_E}\cap\mathcal O_{\mu}\setminus\{0\} \}\cup -\{\omega(x)\mid x\in m_{v_E}\cap\mathcal O_{\mu}   \setminus\{0\}\}
%$$
%which corresponds to the prime ideal $m_{v_E}\cap \mathcal O_{\mu}$ (taking the intersection in $\mathcal O_{v_E})$.
%as explained in Theorem 11.10.15 \cite{ZS2}.

By Theorem VI.10.17 \cite{ZS2}, we have a natural short exact sequence
$$
0\rightarrow \mu L\rightarrow \omega K\rightarrow v_E K\rightarrow 0,
$$
where $\mu L\cong \ZZ$ and $v_EK\cong  \ZZ$. Thus $\omega K\cong \ZZ^2$, with the lex order.  By construction, $\omega$ has the center $\alpha$ on $Z(\mathcal I)$.

Let $q$ be the center of $\omega$ on $X$. Then $q\in E$ and $\mathcal O_{X,q}$ is dominated by $\omega$.  Thus $\beta=q$. By Lemma \ref{Lemma22}, $q\in \mbox{sBL}(-\Delta)$. We have earlier shown that $E\not\subset \mbox{sBL}(-\Delta)$. Since $\mbox{sBL}(-\Delta)$ is a simple normal crossings divisor which is a union of exceptional curves $E_i$ by Theorem \ref{Theorem2} 3), and these $E_i$ are nonsingular, there exists a nonsingular projective curve $F\subset \mbox{sBL}(-\Delta)$ over $k=R/m_R$ which contains $q$. Thus in the two dimensional regular local ring $\mathcal O_{X,q}$, we have regular parameters $s,t$ such that $s=0$ is a local equation of $F$ at $q$ and $t=0$ is a local equation of $E$ at $q$. Given $g\in \mathcal O_{X,q}$, let $\overline g$ be the residue of $g$ in $\mathcal O_{X,q}/(t)=\mathcal O_{E,q}$. We have that $\overline s$ is a regular parameter in $\mathcal O_{E,q}$. Thus for $f\in \mathcal O_{X,q}$, writing $f=t^ag$ where $g\in \mathcal O_{X,q}$ is such that $t\not | g$, we have that
$\omega(f)=(a,\mu(\overline g))$. In particular, we have that $\omega(s)=(0,1)$.

Expand $\Delta=\sum a_iF_i$ where without loss of generality, $F_1=E$ and $F_2=F$. There exists a bound $0<b$  and an ideal  $J_n\subset \mathcal O_{X,q}$ such that $J_n$ is an $m_q$-primary ideal or $J_n=\mathcal O_{X,q}$, such that  
$$
\Gamma(X,\mathcal O_X(-n\Delta))\mathcal O_{X,q}=t^{a_1n}s^{a_2n+c_n}J_n
$$
for $n\ge 1$, where $0<c_n<b$, since $E\not\subset \mbox{sBL}(-\Delta)$ and $F\subset \mbox{sBL}(-\Delta)$. 
This follows since 
$$
a_2n+c_n=v_F(\Gamma(X,\mathcal O_X(-n\Delta)\mathcal O_{X,q})=v_F(I_n)
$$
where $v_F$ is the valuation associated to $F$. The inequality $0<c_n$ then follows from Proposition \ref{Prop8} and $c_n<b$ follows from  
Proposition 5.3 \cite{CNg}.

The short exact sequences
$$
0\rightarrow \mathcal O_X(-n\Delta-E)\rightarrow\mathcal O_X(-n\Delta)\rightarrow \mathcal O_X(-n\Delta)\otimes\mathcal O_E\rightarrow 0
$$
 induce  maps 
 $\Gamma(X,\mathcal O_X(-n\Delta))\mathcal O_{X,q}\rightarrow \mathcal O_X(-n\Delta)\otimes\mathcal O_{E,q}$ with image 
 $$
 s^{c_n}J_n\mathcal O_X(-n\Delta)\otimes\mathcal O_{E,q}\cong
 s^{c_n}J_n\mathcal O_{X,q}/(t).
 $$
  $\overline s$ is a regular parameter in the 1 dimensional regular local ring $\mathcal O_{E,q}$, so $s^{c_n}J_n\mathcal O_{X,q}/(t)=\overline s^{c_n+d_n}(\mathcal O_{X,q}/(t))$ for some $d_n\ge 0$. That is, 
$$
\Lambda_n := \mbox{Image}(\Gamma(X,\mathcal O_X(-n\Delta))\rightarrow \Gamma(X,\mathcal O_X(-n\Delta))\otimes\mathcal O_E))
$$
is contained in $\Gamma(X,\mathcal O_X(-n\Delta)\otimes\mathcal O_E(-(c_n+d_n)q))$. We have  natural short exact sequences 
$$
0\rightarrow \Lambda_n\rightarrow \Gamma(X,\mathcal O_X(-n\Delta)\otimes\mathcal O_E)\rightarrow H^1(X,\mathcal O_X(-n\Delta-E)).
$$
By Corollary 5.7 \cite{CNg}, since $-\Delta$ is nef, there exists $e\ge 1$ such that 
$$
\ell_R(H^1(X,\mathcal O_X(-n\Delta-E)))\le e
$$
for $n\ge 0$. Thus
\begin{equation}\label{eq50}
U_n:=\dim_{R/m_R}\Gamma(E,\mathcal O_X(-n\Delta)\otimes\mathcal O_E)/\Gamma(E,\mathcal O_X(-n\Delta)\otimes\mathcal O_E(-(c_n+d_n)q))\le e
\end{equation}
for all $n\ge 1$. $E$ is a projective curve over the field $k:=R/m_R$. Let $\kappa(q)=\mathcal O_{E,q}/m_q$. Let 
$$
\sigma(n)=\frac{-n(\Delta\cdot E)-(2p_a(E)-1)}{[\kappa(q):k]}
$$
for $n\ge 0$. Let $n_0:=2p_a(E)-2$. There exists $n_1$ such that $\lfloor \sigma(n)\rfloor>e$ for $n\ge n_1$ (where $\lfloor x\rfloor$ is the greatest integer less than or equal to a real number $x$).
Let $n_2=\max\{n_0,n_1\}$. Then for $n>n_2$,
$$
\chi(\mathcal O_X(-n\Delta)\otimes\mathcal O_E)=h^0(\mathcal O_X(-n\Delta)\otimes\mathcal O_E)
$$
by equation (\ref{eq43}). If $c_n+d_n\ge \sigma(n)$, then
$$
\begin{array}{lll}
h^0(\mathcal O_X(-n\Delta)\otimes\mathcal O_E(-(c_n+d_n)q))&\le& h^0(\mathcal O_X(-n\Delta)\otimes\mathcal O_E(-\lfloor \sigma(n)\rfloor q))\\
&=&\chi(\mathcal O_X(-n\Delta)\otimes\mathcal O_E(-\lfloor \sigma(n)\rfloor q))\\
&=& -n(\Delta\cdot 
E)-\lfloor \sigma(n)\rfloor[\kappa(q):k]+1-p_a(E)
\end{array}
$$
by (\ref{eq43}) and (\ref{eq44}).
Thus by (\ref{eq44}),
$$
\begin{array}{lll}
e&\ge &U_n=h^0(\mathcal O_X(-n\Delta)\otimes\mathcal O_E)-h^0(\mathcal O_X(-n\Delta)\otimes\mathcal O_E(-(c_n+d_n)q))\\
&\ge& 
\chi(\mathcal O_X(-n\Delta)\otimes\mathcal O_E)-[-n(\Delta\cdot E)-\lfloor \sigma(n)\rfloor [\kappa(q):k]+1-p_a(E)]\\
&=& -n(\Delta\cdot E)+1-p_a(E)+n(\Delta\cdot E)+\lfloor \sigma(n)\rfloor [\kappa(q):k]-1+p_a(E)\\
&=& \lfloor \sigma(n)\rfloor [\kappa(q):k]>e
\end{array}
$$

which is impossible. So $c_n+d_n<\sigma(n)$. Thus
$$
\chi(\mathcal O_X(-n\Delta)\otimes\mathcal O_E(-(c_n+d_n)q))=h^0(\mathcal O_X(-n\Delta)\otimes\mathcal O_E(-(c_n+d_n)q))
$$
by (\ref{eq43}) and the definition of $\sigma(n)$.

By the Riemann Roch Theorem, equations (\ref{eq4*}) and (\ref{eq44}),
$$
U_n=\chi(\mathcal O_X(-n\Delta)\otimes\mathcal O_E)-\chi(\mathcal O_X(-n\Delta))\otimes\mathcal O_E(-(c_n+d_n)q))=(c_n+d_n)[\kappa(q):k].
$$

so by (\ref{eq50}),
$$
U_n=(c_n+d_n)[\kappa(q):k]\le e 
$$
which implies that
$$
c_n+d_n\le\frac{e}{[\kappa(q):k]}
$$
for $n\ge n_2$. Now $\omega(I_n)=(na_1,na_2+c_n+d_n)$ and $\frac{\omega(I_n)}{n}=(a_1,a_2+\frac{c_n+d_n}{n})$, so 
$\gamma_{\omega}(\mathcal I)=(a_1,a_2)$, and since $c_n>0$ for all $n$, $\frac{\omega(I_n)}{n}\ne \gamma_{\omega}(\mathcal I)$
for all $n>n_2$, a contradiction to Proposition \ref{Prop6}. Thus for a valuation $v$ of $K$ which has a center in $\mbox{sBL}(-\Delta)$ in $X$, we have that $v$ does not have a center in $Z(\mathcal I)$.
\end{proof}

\subsection{Proofs of Theorem 6 and  Theorem 7}\label{SubProof}
 
The proof of Theorem \ref{Theorem6} now follows from Theorem \ref{Theorem4},  Theorem \ref{Theorem5} and the 
discussion before the statement of Theorem \ref{Theorem4}.

We now prove Theorem \ref{Theorem7}. The fact that 1) is equivalent to 3) is by Proposition \ref{Prop6}. 2) implies 1) follows from  Theorem \ref{Theorem4} and 1) implies 2) is Theorem \ref{Theorem5}.

\end{document}